\documentclass{amsart}

\usepackage{amssymb, amsfonts, amsthm, amsmath, latexsym}
\usepackage[all]{xy}

\def\bt{\begin{thm}}
\def\et{\end{thm}}
\def\bl{\begin{lem}}
\def\el{\end{lem}}
\def\bd{\begin{defi}}
\def\ed{\end{defi}}
\def\bc{\begin{cor}}
\def\ec{\end{cor}}
\def\bp{\begin{proof}}
\def\ep{\end{proof}}
\def\br{\begin{rem}}
\def\er{\end{rem}}

\newtheorem{thm}{Theorem}[section]
\newtheorem{prop}[thm]{Proposition}
\newtheorem{lem}[thm]{Lemma}
\newtheorem{defn}[thm]{Definition}
\newtheorem{example}[thm]{Example}
\newtheorem{rem}[thm]{Remark}
\newtheorem{cor}[thm]{Corollary}

\numberwithin{equation}{section}

\newcommand{\fraction}[2]{\frac{\textstyle #1}{\textstyle #2}}
\newcommand{\vmin}{v_{\alpha}^{min}}
\newcommand{\Tmin}{T_{\alpha}^{min}}
\newcommand{\cohomology}{H^{1,1}(X,\mathbb{R})}
\newcommand{\psef}{H^{1,1}_{psef}(X,\mathbb{R})}
\newcommand{\nef}{H^{1,1}_{nef}(X,\mathbb{R})}
\newcommand{\bigc}{H^{1,1}_{big}(X,\mathbb{R})}

\newcommand{\inter}{\alpha \cdot C}

\newcommand{\bthm}{\begin{thm}}
\newcommand{\ethm}{\end{thm}}
\newcommand{\bstp}{\begin{stp}}
\newcommand{\estp}{\end{stp}}
\newcommand{\blemma}{\begin{lemma}}
\newcommand{\elemma}{\end{lemma}}
\newcommand{\bprop}{\begin{prop}}
\newcommand{\eprop}{\end{prop}}
\newcommand{\bpf}{\begin{pf}}
\newcommand{\epf}{\end{pf}}
\newcommand{\bdefn}{\begin{defn}}
\newcommand{\edefn}{\end{defn}}
\newcommand{\brk}{\begin{rmrk}}
\newcommand{\erk}{\end{rmrk}}
\newcommand{\bcrl}{\begin{crl}}
\newcommand{\ecrl}{\end{crl}}

\title{Green Currents for Meromorphic Maps of Compact K\"ahler Manifolds}
\author{Turgay Bayraktar}
\date{February 9, 2012}
\address{Mathematics Department, Indiana University 47405 Indiana, USA}
\email{tbayrakt@indiana.edu}

\begin{document}

\maketitle

\section{Introduction}

Let $X$ be a compact K\"ahler manifold and $f:X\dashrightarrow X$ be a dominant meromorphic map.  It is known that we may define a linear pullback map $f^*:H^{1,1}(X,\mathbb{R}) \rightarrow H^{1,1}(X,\mathbb{R}).$ However, in general this linear action is not compatible with the dynamics of the map $f$. We say that $f$ is 1-regular  whenever $(f^n)^*=(f^*)^n$ for $n=1,2,\dots$ on $\cohomology.$ In the sequel we will assume that $f$ is 1-regular. By a standard Perron-Frobenius type argument there exists $\alpha \in  \psef$ such that  
\begin{equation*}
f^*\alpha=\lambda_1(f) \alpha
\end{equation*}
 where $\lambda_1(f)$ is defined to be the spectral radius of $f^*.$ Let 
 $$H:=\{\alpha \in H^{1,1}(X,\mathbb{R}):f^*\alpha=\lambda_1(f) \alpha\}$$
 we also consider
 $$H_\mathcal{N}:=\{\alpha\in  H^{1,1}(X,\mathbb{R}): \alpha=\lim_{N\rightarrow \infty}\frac{1}{N}\sum_{n=1}^N \frac{1}{n^{m-1}\lambda_1^n}(f^n)^*\beta \ \text{for some}\ \beta \in \nef\}$$
 where $m$ denotes 
 the size of the largest Jordan block associated to $\lambda_1(f).$
 Then it follows that $H_\mathcal{N}\subset H \cap \psef$ and  $H_{\mathcal{N}}$ has a non-empty interior in $H.$ 
 \\ \indent 
 In general, a class $\alpha \in H_\mathcal{N}$ is not numerically effective (nef). Boucksom \cite{Bou} has defined the minimal multiplicity $\nu(\alpha,x)$ of a class $\alpha \in \psef$ at a point $x\in X.$ This is a local obstruction to the numerical effectiveness of $\alpha \in \psef$ at $x$. The set 
$$E_{nn}(\alpha):=\{x\in X: \nu(\alpha,x)>0\}$$
 is called the non-nef locus of $\alpha.$ A property of $E_{nn}(\alpha)$ is that if $C\subset X$ is an irreducible algebraic curve such that $\alpha \cdot C<0$ then $C\subset E_{nn}(\alpha).$ We let $I_{f^k}$ denote the indeterminacy locus of the iterate $f^k$.     
\begin{thm}
Let $f:X\dashrightarrow X$ be a dominant meromorphic map. If $f$ is 1-regular and $\lambda_1(f)>1$ then $\displaystyle E_{nn}(\alpha) \subset  \bigcup_{k=1}^{\infty} I_{f^k}$ for every $\alpha\in H_\mathcal{N}.$
\end{thm}  
\noindent As a corollary we obtain that every curve $C$ such that $\inter <0$ is a subset of $\displaystyle\bigcup_{k=1}^{\infty} I_{f^k}$. Moreover, the non-nef locus $E_{nn}(\alpha)$ does not contain any hypersurface of ~$X$. \\ \indent
 Many authors have constructed positive closed invariant currents to represent the invariant classes. These constructions, however, assume that the class is nef or sometimes even K\"ahler. Here we consider some cases where the invariant class is merely psef.\\ \indent
  Let us fix a smooth representative $\theta \in \alpha.$ We say that an upper semi-continuous function $\phi \in L^1(X)$ is a $\theta$-psh function if $\theta+dd^c\phi\geq 0$ in the sense of currents. Following \cite{DPS} we define
 $$v_{\alpha}^{min}:=sup\{\phi \leq 0: \phi\ \text{is}\ \theta\text{-psh function} \}.$$ 
Thus, $\theta +dd^cv_{\alpha}^{min} \in \alpha$ is a positive closed $(1,1)$ current with minimal singularities. 
 \begin{thm} \label{B}
Let $f:X\dashrightarrow X$ be a 1-regular dominant meromorphic map and $\alpha \in H^{1,1}_{psef}(X,\mathbb{R})$ such that $f^*\alpha =\lambda \alpha$ for some $\lambda > 1$. If 
\begin{equation*} 
\frac{1}{\lambda^{n}} v_{\alpha}^{min} \circ  f^{n} \rightarrow  0 \ \text{in} \ L^1(X) \tag{$\star$}
\end{equation*}
then  for every smooth form $\theta \in \alpha$ we have the existence of the limit $$T_{\alpha}:=\lim_{n\rightarrow \infty}\frac{1}{\lambda^n}(f^{n})^*\theta$$ which depends only on the class $\alpha$. $T_{\alpha}$ is a positive closed $(1,1)$ current satisfying $f^*T_{\alpha}= \lambda T_{\alpha}$. Furthermore, 
\begin{itemize}
\item[(1)] $T_{\alpha}$ is minimally singular among the invariant currents which belong to the class $\alpha$.
\item[(2)] $T_{\alpha}$ is extreme within the cone of positive closed (1,1) currents whose cohomology class belongs to $\mathbb{R}^+\alpha$.
\end{itemize}
 \end{thm} 
 We have seen that such $\alpha$ exists for $\lambda=\lambda_1(f)$ but we can also allow other values of $\lambda$ as well. We also prove that $(\star)$ is a necessary condition under natural dynamical assumptions (see Proposition \ref{A}). \\ \indent
 The following result provides and algebraic criterion for the existence of Green currents when $X$ is projective:
 \begin{thm}
Let $X$ be a projective manifold and $f:X\dashrightarrow X$ be a dominant 1-regular rational map. Assume that $\lambda:=\lambda_1(f)>1$ is a simple eigenvalue of $f^*$ with $f^*\alpha_f =\lambda \alpha_f$. If $ \alpha_f \cdot C \geq 0$ for every algebraic irreducible curve $C \subset E_f^-:=f(I_f)$ then
\begin{equation*} 
\frac{1}{\lambda^{n}} v_{\alpha}^{min} \circ  f^{n} \rightarrow  0 \ \text{in} \ L^1(X). 
\end{equation*}  
 \end{thm}  
\indent In the last part of this work, we present some examples of birational maps in higher dimensions which fall into the frame work of Theorem 1.2, nevertheless the invariant class is not nef.\\ \indent
 Let $f:=L \circ J$ where $J:\mathbb{P}^d \dashrightarrow \mathbb{P}^d$ 
$$J[x_0:x_1:\dots:x_d]=[x_0^{-1}:x_1^{-1}:\dots:x_d^{-1}]$$
and $L$ is a linear map given by  given by a $(d+1)\times (d+1)$ matrix of the form 
\begin{equation*}
 L=
 \begin{bmatrix}
a_0-1 & a_1 & a_2 & \ldots & a_d
 \\
a_0 & a_1-1 & a_2 & \ldots & a_d
 \\
a_0 & a_1 & a_2-1 & \ldots & a_d
\\
\vdots & \vdots & \vdots & \ddots & \vdots \\

a_0 & a_1 & a_2 & \ldots & a_d-1
\end{bmatrix}
\end{equation*}
 with $a_j \in \mathbb{C}$ and $\sum_{j=0}^d a_j=2.$ The linear map $L$ is involutive that is $L =~ L^{-1}$ in $PGL(d+1,\mathbb{C}).$ Let $\Sigma_i:=\{[x_0:\dots:x_d]\in \mathbb{P}^d: x_i=0\}$ then $p_i:=f(\Sigma_i)\in~\mathbb{P}^d$ is the $i^{th}$ column of the matrix $L$. We define its orbit $\mathcal{O}_i$ as follows: $\mathcal{O}_i=\{p_i,f(p_i),f^2(p_i),\dots,f^{N_i-1}(p_i)\}$ if $f^j(p_i)\not\in I_f$ for $0\leq j\leq N_i-2$ and $f^{N_i-1}(p_i) \in I_f$ for some $N_i \in \mathbb{N}$, otherwise $\mathcal{O}_i=\{p_i,f(p_i),f^2(p_i),\dots\}.$ If $\mathcal{O}_i$ is finite (the first case above) we say that the orbit of $\Sigma_i$ is singular of length $N_i$. It follows from \cite{BK} that there exists a complex manifold $X$ together with a proper modification $\pi:X \rightarrow \mathbb{P}^d$ such that the induced map $f_X:X\dashrightarrow X$ is 1-regular. Moreover, if the length of the singular orbits are long enough (see Theorem \ref{BK}) then $\lambda_1(f)>1$ is the unique simple eigenvalue of $f^*|_{H^{1,1}(X,\mathbb{R})}$ of modulus greater than one. We define $S:=\{i\in\{0,1,\dots,d\}|\ \mathcal{O}_i\ \text{is singular}\}$ and denote its cardinality by $|S|$. If $S$ is non-empty, by conjugating $f$ with an involution without lost of generality we may assume that $S=\{0,\dots,k\}.$ 
\begin{thm}
Let $f_X:X \dashrightarrow X$ be as above with $\lambda:=\lambda_1(f_X)>1$ and $\alpha_f \in \psef$ such that $f^*\alpha_f =\lambda \alpha_f$. Then $\alpha_f$ is nef if and only if $|S|\leq1.$ Moreover, if $2 \leq|S|\leq d$ and all singular orbits of $f$ have the same length then 
\begin{equation*}
E_{nn}(\alpha_f)=
\begin{cases}
\{[x_0:x_1:\dots:x_d]:\ x_i=0\ \text{for}\ k+1\leq i\leq d\} & \text{if } k \leq d-2\\
\displaystyle \bigcup^{d-1}_{i=0} \{[x_0:x_1:\dots:x_d]:\ x_{i}=x_d=0\} & \text{if } k=d-1
\end{cases}
\end{equation*}
\end{thm}  
We also show that these maps fall into frame work of Theorem 1.2:
\begin{thm}\label{B2}
Let $f_X:X \dashrightarrow X$  be as above  with $\lambda_1(f_X)>1$. If $a_i\not=0$ for every $i\in S$ then condition $(\star)$ in Theorem 1.2 holds.
\end{thm}
The outline of the paper as follows. In section 2, we provide the basic definitions and results which we will use in the sequel. In section 3, we discuss invariance properties of closed convex cones in $H^{1,1}(X,\mathbb{R})$ and prove Theorem 1.1. Section 4 is devoted to the proof of Theorem 1.2. We also discuss some cases for which $(\star)$ holds in section 4. In section 5, we discuss rational maps and prove Theorem 1.3. In the last section, we prove Theorems 1.4 and 1.5.
\section*{Acknowledgement}
I would like to express my sincere thanks to my advisor E. Bedford for his guidance and interest on this work. I am also grateful to J. Diller and V. Guedj for their comments and suggestions on an earlier draft. I also thank to the referee for his suggestions. 
\section{Preliminaries}

\subsection{Positive Cones}
Let $X$ be a compact K\"ahler manifold of dimension $k$ and $\omega$ be a fixed K\"ahler form  satisfying $\int_X\omega^k=1.$  All volumes will be computed with respect to the probability volume form $dV:=\omega^k$. Let $H^{1,1}(X)$ denote the Dolbeault cohomology group and let $H^2(X,\mathbb{Z})$, $H^2(X,\mathbb{R})$ and $H^2(X,\mathbb{C})$ denote the de-Rham cohomology groups with coefficients in $\mathbb{Z}, \mathbb{R}, \mathbb{C}.$ We also set $$H^{1,1}(X,\mathbb{R}):=H^{1,1}(X) \cap H^2(X,\mathbb{R}).$$ 

\begin{defn}
 A class $\alpha \in H^{1,1}(X,\mathbb{R)}$ is called K\"ahler if $\alpha$ can be re\-presented by a K\"ahler form.We denote the set of all K\"ahler classes by $\mathcal{K}$.  A class $\alpha$ is called numerically effective (nef) if it lies in the closure of the K\"ahler cone. The set of all nef classes will be denoted by $\nef.$
\end{defn}
 An upper semi continuous function $\varphi \in L^1(X)$ is called quasi-plurisubharmonic (qpsh) if there exists a smooth closed form $\theta$ such that $\theta + dd^c \varphi \geq 0$ in the sense of currents. Notice that a qpsh function is locally sum of a smooth function and a psh function. A closed (1,1) current T is called almost positive if there exists a real smooth (1,1) form $\gamma$ such that $T \geq \gamma$. \\ \indent
 The \textit{Lelong number} of a positive closed (1,1) current $T$ is defined by 
 $$ \nu(T,x):=\displaystyle \liminf_{z\rightarrow x}\frac{\phi(z)}{\log|x-z|}$$
  where $\phi$ is a local potential for $T$ that is $T=dd^c \phi$ near $x.$ This definition is independent of the choice of the potential $\phi$ and the local coordinates. If $T$ is almost positive then the Lelong numbers are still well-defined since the negative part contributes for zero. It follows from a theorem of Thie that $\nu([D],x)=mult_xD$ where $[D]$ is the current of integration along an effective divisor and $mult_x$ is the multiplicity of $D$ at $x.$ We denote the sub-level sets by $E_c(T):=\{x \in X: \nu(T,x)\geq c\}.$ A Theorem of Siu asserts that $E_c(T)$ is an analytic set of codimension at least 1. We also set $$E_+(T):=\cup_{c>0}E_c(T).$$

\indent  A class $\alpha \in H^{1,1}(X,\mathbb{R})$ is called \textit{pseudo-effective} (psef) if there exists a positive closed $(1,1)$ current $T$ such that $\{T\}=\alpha.$ The set of all psef classes, $\psef$ is a closed convex cone containing $\nef$. A positive closed current $T$ is called K\"ahler if there exists small $\epsilon>0$ such that $T \geq \epsilon \omega.$ A class $\alpha \in H^{1,1}(X,\mathbb{R})$ is said to be \textit{big} if there exists a K\"ahler current $T$ such that $\alpha=\{T\}.$ We denote the set of all big classes by $H_{big}^{(1,1)}(X,\mathbb{R})$. This is an open convex cone and coincides with the interior of $\psef$. Finally, we stress that these definitions coincide with the classical ones in complex geometry \cite{Dem}. 
 \begin{thm}[\cite{Bou, DP}] 
A class $\alpha \in \nef$ is big if and only if $\alpha^n \neq 0$.
\end{thm} 
 
 \subsection{Currents with analytic singularities}

\indent Following \cite{Dem} and \cite{Bou}, a closed almost positive (1,1) current $T=\theta+dd^c\phi$ is said to have analytic singularities along a subscheme $V(\mathcal{I})$ defined by a coherent ideal sheaf $\mathcal{I}$  if there exists $c>0$ and locally
$$\phi=\frac{c}{2}\log\big(|f_1|^2 +..+|f_N|^2\big)+u$$ 
where $u$ is a smooth function and $f_1,..,f_N$'s are holomorphic functions which are local generators of $\mathcal{I}.$ Blowing-up $X$ along $V(\mathcal{I})$ and resolving the singularities in the sense of Hironaka, we obtain a modification $\mu:\widetilde{X} \rightarrow X.$ Moreover, $D:=\mu^{-1}(V(\mathcal{I}))$ is an effective divisor in $\widetilde{X}$ and $\mu^*T$ has analytic singularities along $D$, thus it follows from Siu decomposition that 
$$\mu^*T=\theta+cD$$
where $\theta$ is a smooth (1,1) form. Furthermore, if $T\geq \gamma$ then we have $\theta \geq~\mu^*\gamma$. In particular, if $T\geq0$ then $\theta \geq0$.   
This decomposition is called \textit{log resolution} of singularities of $T$.
 \begin{thm}[\cite{Dem}] \label{Dem}
 Let $T\geq \gamma$ be an almost positive closed (1,1) current on $X$. Then there exists  a sequence of positive real numbers $\epsilon_n$ decreasing to 0 and a sequence of almost positive closed (1,1) currents $T_n \in \{T\}$ with analytic singularities such that $T_n \rightarrow T$ weakly, $T_n\geq \gamma-\epsilon_n \omega $ and $\nu(T_n,x)$ increases uniformly to $\nu(T,x)$ with respect to $x \in X$.
 \end{thm}
 
 \subsection{Currents with minimal singularities}
  
  Let $\varphi_1$ and $\varphi_2$ be two qpsh functions. Following \cite{DPS}, we say that $\varphi_1$ is less singular than $\varphi_2$ if $\varphi_2 \leq \varphi_1+C$ for some constant $C.$ If $T_1$ and $T_2$ are two closed almost positive currents we write $T_i=\theta_i + dd^c \varphi_i$ where $\theta_i \in \{T_i\}$ is a smooth closed form and $\varphi_i$ is a qpsh function. We say that $T_1$ is less singular than $T_2$ if $\varphi_2 \leq \varphi_1+C$ . Notice that this definition is independent of the choice of the representatives $\theta_i$ and potentials $\varphi_i$. \newline
\indent  For a class $\alpha \in H^{1,1}(X,\mathbb{R})$ and a real smooth (1,1) form $\gamma$, we denote the set of all closed almost positive (1,1) currents $T \in \alpha$ satisfying $T \geq \gamma$ by~$\alpha[\gamma]$. We fix a smooth represantative $\theta \in \alpha$ and  define 
  $$ v^{min}_{\alpha,\gamma}:= sup\{\varphi \leq 0\ |\ \theta+dd^c \varphi \geq \gamma\}.$$
It follows that $T^{min}_{\alpha, \gamma}:= \theta + dd^c  v^{min}_{\alpha,\gamma} \in \alpha[\gamma]$  and $\nu(T^{min}_{\alpha, \gamma},x) \leq~ \nu(T,x)$ for every $x \in X$ and for every $T\in\alpha[\gamma]$. If, in particular $\gamma=0$ then we write $T^{min}_{\alpha,\gamma}=\Tmin$ and refer to it as the \textit{minimally singular current}. Notice that minimally singular currents are not unique in general. For example, if $\alpha \in \mathcal{K}$ then every smooth positive closed form $\theta \in \alpha$ is a minimally singular current. However, if $S= \theta' + dd^c u^{min}_{\alpha,\gamma} \in \alpha$ is another such current, since $v^{min}_{\alpha,\gamma}-u^{min}_{\alpha,\gamma} \in L_{loc}^{1}(X)$ and $dd^c\big(v^{min}_{\alpha,\gamma}-u^{min}_{\alpha,\gamma}\big)$ is  smooth, $v^{min}_{\alpha,\gamma}-u^{min}_{\alpha,\gamma}$ is also smooth and hence, bounded. Thus,  $v^{min}_{\alpha,\gamma}$ and $u^{min}_{\alpha,\gamma}$ are equivalent in the sense of singularities. Therefore, for a fixed class $\alpha \in \psef$, the current of minimal singularities is well-defined modulo $dd^c(C^{\infty})$.  

\subsection{Minimal multiplicities and non-nef locus}
Following \cite{Bou}, for a class $\alpha \in \psef$ we define 
$$ \nu(\alpha,x):=\displaystyle \sup_{\epsilon>0}\nu(T^{min}_{\alpha,\epsilon},x) $$
where $T^{min}_{\alpha,\epsilon}:=T^{min}_{\alpha, -\epsilon \omega} \in \alpha[-\epsilon \omega]$. Since the right hand side in the definition of $ \nu(\alpha,x)$ is increasing as $\epsilon$ decreases, the $\sup$ coincides with the limit.  This definition is independent of the choice of the K\"ahler form $\omega$. We also remark that for every positive closed (1,1) current $T \in \alpha$ and  $x \in X,$ we have $0\leq \nu(\alpha,x) \leq \nu(T,x) \leq C$ where $C >0$ is a constant depending only on the cohomology class $\alpha$. 
\\ \indent
 If $A$ is an analytic subset of X we define
$$ \nu(\alpha,A):=\displaystyle \inf_{x\in A}\nu(\alpha,x).$$ 
\indent A class $\alpha \in \psef$ is called \textit{nef in codimension 1} if $\nu(\alpha,D)=0$ for every prime divisor $D\subset X.$ We denote the set of all such classes by $\mathcal{E}_1.$ It follows from the following proposition that $\mathcal{E}_1 \subset \psef$ is also a closed convex cone.
\begin{prop}[\cite{Bou}]\label{Bou}
Let $\alpha \in \psef$ be a class
\begin{itemize}
\item[(i)] $\alpha$ is nef if and only if $\nu(\alpha,x)=0$ for every $x \in X.$
\item[(ii)] $\alpha \rightarrow \nu(\alpha,x)$ is sub-additive and homogenous in $\alpha$ for every $x \in X.$
\item[(iii)]$\alpha \rightarrow \nu(\alpha,x)$ is lower semi-continuous on $\psef$ and continuous on $\bigc$ for every $x\in X$. 
\item[(iv)] If $\alpha\in \bigc$ then $\nu(\alpha,x)= \nu(\Tmin,x)$ for every $x \in X.$
\end{itemize}
\end{prop}
\begin{cor}
If $\alpha \in \nef \cap H^{1,1}_{big}(X,\mathbb{R})$ then $\nu(\Tmin,x)=0$ for every $x \in X.$ Moreover, if $\alpha \in \psef$ and $\nu(\Tmin,x)=0$ for every $x\in X$ then $\alpha$ is nef.
\end{cor}
\begin{defn}[\cite{Bou}] Let $\alpha \in \psef$ then non-nef locus of $\alpha$ is defined by 
$$ E_{nn}(\alpha):=\{ x \in X\ | \nu(\alpha,x)>0\}.$$
\end{defn} 
We also have the following description of the non-nef locus:
\begin{prop}
Let $\alpha \in \psef$ then
$$E_{nn}(\alpha)= \displaystyle \bigcup_{\epsilon>0} \bigcap_T \mu(|D|)$$
where $T$ runs over the set $\{T\in \alpha[-\epsilon \omega]: T\ has \ analytic\ singularities\}$ and $\mu:\widetilde{X} \rightarrow X$, $\mu^*T=\theta + [D]$ is log resolution of singularities of $T$ and $|D|$ denotes the support of the current of integration $[D]$.
\end{prop}
\begin{proof}
Let $x \in  \displaystyle \bigcup_{\epsilon>0} \bigcap_T \mu(|D|).$ Then there exists  $\epsilon_1>0$ such that $x \in \mu(|D|)$ for every  $T \in \alpha[-\epsilon_1 \omega]$ which has analytic singularities and log resolution  $\mu:\widetilde{X} \rightarrow X$, $\mu^*T=\theta+[D].$ This implies that  $\nu(T,x)>0$. Let  $\epsilon < \epsilon_1$ and let $T^{min}_{\alpha,\epsilon}$ be current of minimal singularities. By Theorem \ref{Dem} there exists a sequence $T_k \in \alpha[-\epsilon_1 \omega]$ with analytic singularities such that $T_k$ converges weakly to $T$ and $\nu(T_k,x)$ increases to $\nu(T^{min}_{\alpha,\epsilon},x).$ Thus, $0<\nu(T^{min}_{\alpha,\epsilon},x) \leq \nu(\alpha,x).$ That is $x \in E_{nn}(\alpha).$

  To prove reverse inclusion, let $x \in E_{nn}(\alpha),$ then by definition of $\nu(\alpha,x)$ there exists $\epsilon>0$ such that $0<\nu(T^{min}_{\alpha,\epsilon},x).$ Now, let $T\in \alpha[\epsilon]$ has analytic singularities. Resolving it's singularities we obtain $\mu: \widetilde{X} \rightarrow X$ such that $\mu^*T=\theta+[D]$ where $\theta \geq -\epsilon \omega$ is smooth and $D$ is an effective divisor. Since $0<\nu(T^{min}_{\alpha,\epsilon},x) \leq \nu(T,x)\leq \nu(\mu^*(T),p)$ for every $p\in \widetilde{X}$ with $\mu(p)=x,$ we conculde that $x \in \mu(|D|)$.
\end{proof}

\begin{defn}
Let $\alpha \in \psef$ be a class. An irreducible algebraic curve $C$ is called $\alpha$-negative if the intersection product $\inter <0.$
\end{defn} 
\begin{prop} [\cite{BDPP}] \label{negative}
For $\alpha \in \psef$ every $\alpha$-negative curve $C$ is contained in $E_{nn}(\alpha).$  
\end{prop}
\begin{proof}
If $C\not\subset E_{nn}(\alpha)$ then for every $\epsilon>0$ there exists $T\in \alpha[-\epsilon \omega]$ with analytic singularities such that $C \not\subset \mu(|D|)$ where $\mu:\widetilde{X} \rightarrow X$ is log reolution of $T$ and $\mu^*(T)=\theta+[D]$. \newline \indent
Let $\widetilde{C} \subset \widetilde{X}$ be the strict transform of $C$ so that $\mu_*\widetilde{C}=C$ and define  $S=T+\epsilon \omega\geq 0.$ Then $\mu^*S=\theta_{\epsilon} +[D]$ where $\theta_{\epsilon}\geq0$ since $S\geq 0.$ Thus,
$$ \alpha +\epsilon\{ \omega\} \cdot C = \langle S, \mu_*\widetilde{C}\rangle =\langle\mu^*S, \widetilde{C}\rangle=\langle \theta_{\epsilon}+ [D] ,\widetilde{C}\rangle \geq 0$$
the last inequality follows from $\theta_{\epsilon} \geq 0$ and $\widetilde{C} \not\subset D.$ Since $\epsilon >0$ is arbitrary we obtain $\inter \geq 0.$ 

\end{proof}
\begin{rem}
Thus, the non-nef locus contains the union of the $\alpha-$negative curves. However, in general, the non-nef locus of a pseudo-effective class $\alpha$ is not equal to union of $\alpha$-negative curves (see \cite{BDPP}).
\end{rem}

\section{Dynamics of Meromorphic maps}

Let $X$ be a compact K\"ahler manifold of dimension $k$ and $f:X\dashrightarrow X$ be a \textit{meromorphic map} that is $f$ is holomorphic on the set $X\backslash I_f$ such that the closure of the graph $\Gamma_f$ of $f:X\backslash I_f \rightarrow X$ in $X\times X$ is a irreducible analytic set of dimension $k$. Let $\pi_i:X\times X \rightarrow X$ denote the canonical projections. Then $I_f$ coincides with the set of points $z$ for which $\pi^{-1}_1(z)\cap \Gamma_f$ contains more than one point. The set $I_f$, called the \textit{indeterminacy set} of $f,$ is also an analytic set of codimension at least 2. In fact, for every $z\in I_f$ the set $\pi^{-1}_1(z)\cap \Gamma_f$ has positive dimension. Moreover, $X\backslash I_f$ is the largest open set where $f$ is holomorphic. We also set
 $$\mathcal{I_1}:=\bigcup_{n\geq 1} I_{f^n}.$$
\indent For a subset $Z\subset X$ we define the \textit{total transform} of $Z$ under $f$ by 
$$f(Z):=\pi_2(\pi_1^{-1}(Z)\cap \Gamma_f).$$
With the above convention we define $E_f^-:=f(I_f)$.\\ \indent
 We say that $f$ is \textit{dominant} if the projection $\pi_2$ restricted to $\Gamma_f$ is surjective. This is equivalent to saying that Jacobian determinant of $f$ does not vanish identically in any coordinate chart.  We refer the reader to the surveys \cite{Si} and \cite{G} for basic properties of meromorphic maps. 
\\ \indent
It's well known that $f$ induces a linear action on $H^{p,p}(X, \mathbb{R})$ as follows: \\
Let $\widetilde{\Gamma_f}$ denote a desingularization of   $\Gamma_f$ then the diagram 
$$\xymatrix{& \widetilde{\Gamma_f}\ar[dl]_{\pi_1}\ar[dr]^{\pi_2} &\\
X\ar@{-->}[rr]_f & & X}$$
commutes
 . Let $\theta$ be a smooth closed $(p,p)$ form on X then we set $f^*\theta=~ (\pi_1)_*(\pi_2)^*\theta$ where $(\pi_2)^*\theta$ is a smooth form and the later is push-forward as a current. Since pull-back and push-forward commute with $d$ operator, $f^*\theta$ is a $d$-closed $(p,p)$ current. Then we set 
$$ f^*\{\theta \}=\{f^*\theta\}$$
where $\{\theta \}$ denotes de-Rham cohomology class of $\theta$ and $\{f^*\theta\}$ is the de-Rham cohomology class of $f^*\theta$. \\ \indent 
  Similarly, one can define push-forward by $f_*\theta= (\pi_2)_*(\pi_1)^*\theta$. It follows that the action of pull-back on $H^{p,p}(X,\mathbb{R})$ is dual to that of push-forward on  $H^{k-p,k-p}(X,\mathbb{R})$ with respect to intersection product. \\ \indent
  We say that $f$ is  \textit{p-regular} whenever $ (f^n)^*=(f^*)^n \ \text{on} \ H^{p,p}(X,\mathbb{R})$ as linear maps for $n= 1,2,\dots$\\ \indent
  We also denote
  $$ \delta_p(f):= \int_{X \backslash I_f}f^*\omega^p \wedge \omega^{k-p} $$
  and the $p^{th}$ \textit{dynamical degree} of $f$ by  
  $$ \lambda_p(f):= \displaystyle\limsup_{n \rightarrow \infty} [\delta_p(f^n)]^{\frac{1}{n}}. $$
  In particular, if $f$ is p-regular  then $\lambda_p(f)$ coincides with the spectral radius of $f^*|_{H^{p,p}(X,\mathbb{R})}.$  
\subsection{Invariant Classes and Singularities}

We denote the set of all positive closed $(1,1)$ currents by $\mathcal{T}(X).$ Following \cite{G02} we define the pull-back of $T$ as follows: we fix a point $x_0 \in X\backslash I_f$, then locally we can write $T= dd^c u$ for a psh function $u$  near $f(x_0)$ and define $f^*T= dd^c u\circ f$ near $x_0$ which is independent of the choice of the local potential $u$. Therefore, we obtain a well-defined positive closed (1,1) current on $X\backslash I_{f}$. Now, since $I_f$ is an analytic set of codimension at least two it follows form \cite{HaPo} that it extends trivially to a unique positive closed (1,1) current $f^* T$ on $X$. Furthermore, since $f$ is dominant the action $T \rightarrow f^*T$ is continuous with respect to weak topology on postive closed (1,1) currents (\cite{Meo}). Moreover, $\{f^* T\} \in H^{1,1}(X,\mathbb{R})$ is independent of the choice of $T \in \alpha$ for $\alpha \in \psef$ and $f^*\alpha=\{f^*T\} \in \psef$. 
   Notice that $\psef$ is a closed, convex cone which is strict (i.e. $\psef \cap -\psef=\{0\}$). Since $\psef$ is invariant under the linear action $f^*:H^{1,1}(X,\mathbb{R}) \rightarrow H^{1,1}(X,\mathbb{R})$ it follows from a Perron-Frobenius type argument (\cite{DF}) that there exists a class $\alpha \in \psef$ such that $f^* \alpha=r_1(f)\  \alpha$ where $r_1(f)$ is the spectral radius of $f^*|_{H^{1,1}(X,\mathbb{R})}$. In particular, if $f$ is 1-regular then $r_1(f)=\lambda_1(f)$. \\ \indent 
The following argument is adapted from \cite{DS}. We choose a basis for $H^{1,1}(X,\mathbb{C})=H^{1,1}(X,\mathbb{R})\otimes_{\mathbb{R}}\mathbb{C}$ so that the associated matrix of $f^*$ is in Jordan form. Let $J_{\lambda_j,m_j}$ denote its Jordan blocks for $1\leq j\leq r.$ In other words, we can decompose $H^{1,1}(X,\mathbb{C})$ into a direct sum of complex subspaces $E_j$
$$H^{1,1}(X,\mathbb{C})= \bigoplus_{1\leq j\leq r} E_j \ \text{with}\ \dim E_j=m_j \ \text{and}\ \sum_{j=1}^rm_j=h^{1,1}$$
such that the restriction of $f^*$ to $E_j$ is given by the Jordan block $J_{\lambda_j,m_j}.$
Since $f^*$ preserves the psef cone which is a proper cone, we may assume that $\lambda_1=r_1(f)$ and $m:=m_1$ is the \textit{index of the spectral radius}. Moreover, we can also assume that $(|\lambda_j|,m_j)$ is ordered so that either $|\lambda_j|> |\lambda_{j+1}|$ or $|\lambda_j|=|\lambda_{j+1}|$ and $m_j\geq m_{j+1}$ for $1\leq j\leq r.$ Let $\nu$ be the integer such that $|\lambda_j|=\lambda_1$ for $1\leq j\leq \nu.$ Let $\tilde{E_j}$ denote the hyperplane generated by the first $m_j-1$ vectors of the basis of $E_j$ associated to the Jordan form. Then we have $\|(f^*)^nv\|\sim n^{m-1}\lambda_1^n$ for every $v\not\in \tilde{E_1}\oplus \dots \oplus \tilde{E_{\nu}}\oplus E_{\nu+1}\oplus \dots \oplus E_r.$ Notice that this property holds for every $v \in \mathcal{K}$ because given any $v' \in H^{1,1}(X,\mathbb{R})$ we can find $v''\in \mathcal{K}$ and $\sigma\geq 0$ such that $v'=\sigma v-v''.$ We let $F_j$ denote the eigenspace of $f^*_{|_{E_j}},$ which is a complex line. We define 
$$F^{\mathbb{C}}:=F_1\oplus \dots \oplus F_{\nu} \ \text{and} \ H^{\mathbb{C}}:=\bigoplus_{\substack{\lambda_j=\lambda_1\\ m_j=m}}F_j.$$
We also set $F:=F^{\mathbb{C}}\cap H^{1,1}(X,\mathbb{R})$ and $H:=H^{\mathbb{C}}\cap H^{1,1}(X,\mathbb{R}).$ Notice that for any $2\leq j\leq \nu$ there exists a unique $\theta_j\in \mathbb{S}:=\mathbb{R}/ 2\pi\mathbb{Z}$ such that $\lambda_j=\lambda_1\exp(i \theta_j).$ Let $\theta:=(\theta_2,\dots,\theta_{\nu})\in \mathbb{S}^{\nu-1}.$ We let $\Theta$ denote the closed subgroup of $\mathbb{S}^{\nu-1}$ generated by $\theta.$ This is a finite union of real tori. The orbit of each point $\theta'\in \Theta$ under the translation $\theta'\rightarrow \theta'+\theta$ is dense in $\Theta.$ If $\lambda_j=\lambda_1$ for every $2\leq j\leq \nu$ then $F=H$ and $\Theta=\{0\}.$ We also set 
$$\Lambda_N:=\frac{1}{N}\sum_{n=1}^N\frac{(f^*)^n}{n^{m-1}\lambda_1^n}$$
For the proof of the following proposition we refer the reader to \cite{DS}.
\begin{prop}\label{DS}
Assume that $\lambda_1>1.$ Then the sequence $(\Lambda_N)$ converges to a surjective real linear map $\Lambda_{\infty}:H^{1,1}(X,\mathbb{R})\rightarrow H.$ Let $n_i$ be an increasing sequence of positive integers then $(n_i^{1-m}\lambda_1^{-n_i} (f^*)^{n_i})$ converges if and only if $(n_i \theta)$ converges. Moreover, any limit $L_{\infty}$ of  $(n_i^{1-m}\lambda_1^{-n_i} (f^*)^{n_i})$ is a surjective real linear map $L_{\infty}:~H^{1,1}(X,\mathbb{R}) \rightarrow F.$
\end{prop}
We also define $$H_{psef}:= H \cap \psef $$ 
   and 
   $$H_\mathcal{N}:=\Lambda_{\infty}(\nef).$$
   It's clear that $H_\mathcal{N} \subset H_{psef}.$ Moreover, it follows from Proposition \ref{DS} that $H_\mathcal{N}$ has a non-empty interior in $H$. Indeed, since $\Lambda_{\infty}$ is surjective it is open. The K\"ahler cone, $\mathcal{K} \subset H^{1,1}(X,\mathbb{R})$ is also open and $\Lambda_{\infty}(\mathcal{K})$ is contained in $H_\mathcal{N}.$      
  \begin{thm}(\cite{Fa},\cite{Ki})\label{F}
 Let $f:X\dashrightarrow X$ be a dominant meromorphic map and $T$ be a positive closed $(1,1)$ current on $X$. Then  for every $x \in X \backslash I_f$
 $$ \nu(T,f(x)) \leq \nu(f^*T,x) \leq C_{f}\nu(T,f(x)) $$
where $C_f>0$ is constant which does not depend on $T$. 
\end{thm}
   
    \begin{thm} \label{nonnef}
   Let $f:X\dashrightarrow X$ be a dominant meromorphic map. 
 \begin{itemize}
\item[(1)] For every $\alpha \in \psef$ and $p \in X\backslash I_f$
$$\nu(f^*\alpha,p)\leq C_f \nu(\alpha,f(p))$$
where $C_f$ is independent of $\alpha.$ In  particular, $f(E_{nn}(f^*\alpha) \backslash I_f) \subset E_{nn}(\alpha).$ 
\item[(2)] Assume that $f$ is 1-regular and $\lambda:=\lambda_1(f)>1.$ Then $E_{nn}(\alpha) \subset \mathcal{I_1}$ for every $\alpha \in H_\mathcal{N}.$ In particular, $H_\mathcal{N} \subset \mathcal{E}_1.$
\end{itemize}
\end{thm}
\begin{proof}
(1) We assume that $\alpha \in \bigc.$ Let $p \in X\backslash I_f$ and $\Tmin \in~ \alpha$ be a positive closed (1,1) current with minimal singularities. Since $f^*\Tmin \in f^*\alpha$ by definition of $\nu(f^*\alpha,p)$ and by Theorem \ref{F} we have
  \begin{equation*}
 \nu(f^*\alpha,p) \leq \nu(f^*(T_{\alpha}^{min}),p) \leq C_f \nu(\Tmin ,f(p))=C_f\nu(\alpha,f(p))
   \end{equation*}
where the last equality follows from $\alpha \in \bigc$.\\
If $\alpha$ is merely psef then we consider $\alpha_{\delta}:=\alpha+\delta\{ \omega\} \in \bigc$ for $ \delta>0.$ Now, by lower semi-continuity of $\nu(\cdot, p)$ and continuity of $f^*$ we get
$$ \nu(f^*\alpha,p) \leq \displaystyle \liminf_{\delta\rightarrow 0}\nu(f^*\alpha_{\delta},p).$$
Since $\alpha_{\delta}$ is big, by above argument we also have 
$$\nu(f^*\alpha_{\delta},p) \leq C_f\nu(\alpha_{\delta},f(p))$$ 
for $\delta>0.$ 
Then by sub-additivity and homogeneity of $\nu(\cdot,f(p))$ we get
$$\nu(\alpha_{\delta},f(p)) \leq \nu(\alpha,f(p))+\delta\nu\big(\{\omega\},f(p))=\nu(\alpha,f(p)\big)$$
since $\{\omega\}$ is K\"ahler. Therefore, the assertion follows.\\
(2) Let $\alpha \in H_\mathcal{N}$ then $\alpha=\lim_{N\rightarrow \infty}\Lambda_N\beta$ for some $\beta \in \nef$ and $f^*\alpha=\lambda \alpha.$  By part(1) we have
$$\nu(\frac{1}{n^{m-1}\lambda^n}(f^n)^*\beta,p)=0$$
for $p\not\in I_{f^n}.$ It follows from sub-additivity, lower semi-continuity of $\nu(\cdot,p)$ and continuity of $f^*$ that 
$$\nu(\alpha,p)\leq \liminf_{N\rightarrow \infty}\frac{1}{N}\sum_{n=1}^N \frac{1}{n^{m-1}\lambda^n}\nu((f^n)^*\beta,p).$$
Thus, $\nu(\alpha,p)=0$ for $p\not\in \mathcal{I_1}.$
\\ Finaly, since $f$ is 1-regular $ \mathcal{I_1}$ does not contain any divisors hence, $\alpha \in \mathcal{E}_1.$

       \end{proof}
       Without the assumption $\alpha \in H_\mathcal{N}$ the assertion of Theorem \ref{nonnef} (2) is not true in general. The following example was communicated by V. Guedj. 
       \begin{example}
       Let $f:\mathbb{P}^2\rightarrow \mathbb{P}^2$ be a holomorphic map of degree $\lambda \geq2$ with a totally invariant point $p$ i.e. $f^{-1}(p)=p$. We define $\pi:X\rightarrow \mathbb{P}^2$ to be the blow up of $\mathbb{P}^2$ at $p.$ Let $f_X$ denote the induced map and $E:=\pi^{-1}(p)$ denote the exceptional fiber. Then $f_X^*\{E\}=\lambda\{E\}.$ Thus, $f^*_X:H^{1,1}(X,\mathbb{R})\rightarrow H^{1,1}(X,\mathbb{R})$ is given by 
     \begin{equation*}
       f_X^*=
       \begin{bmatrix}
       \lambda & 0 \\
       0 & \lambda
       \end{bmatrix}
       \end{equation*}
       Notice that the class $\{E\} \in \psef$ but $ E\cdot E=-1,$ hence, $\{E\}\not\in H_{\mathcal{N}}=\nef.$ 
       \end{example}
       We also remark that in the above case $\lambda_2(f)=\lambda_1(f)^2.$ It follows from \cite{DF} that if $\dim_{\mathbb{C}}(X)=2$ and $r_1(f)^2>\lambda_2(f)$ then $r_1(f)$ is a simple root of the characteristic polynomial of $f^*.$ \\ 
       \indent The following follows from Proposition \ref{negative}:        
\begin{cor}
Let $f$ and $\alpha \in H_\mathcal{N}$ be as in Theorem \ref{nonnef} (2). Then every $\alpha$-negative curve $C \subset \mathcal{I_1}.$
\end{cor}
       \begin{cor}\label{S}
Let $f:X\dashrightarrow X$ be a dominant meromorphic map such that $\dim(I_f)=0$ then $f^*(\nef)\subset \nef.$ In particular, if $X$ is a compact K\"ahler surface $f^*$ preserves the nef cone.       
       \end{cor}
\begin{proof}
Let $\alpha\in \nef$ then it follows from Theorem \ref{nonnef} that $E_{nn}(f^*\alpha) \subset I_f$ and since $\dim(I_f)=0$, $E_{nn}(\alpha)$ is a finite set. Thus, the assetion follows from the regularization argument of \cite[Lemma 6.3]{Dem}.  
\end{proof}      
 If $X$ is a compact K\"ahler surface then the cone $\mathcal{E}_1$ coincides with $\nef$ (\cite{Bou}). Thus, $f^*(\mathcal{E}_1) \subset \mathcal{E}_1$ when $\dim_{\mathbb{C}}X=2$.\\
   However, if $\dim_{\mathbb{C}}X\geq 3$ this is no longer true as the following example shows: \\    

\textbf{Example:} Let
$$J:\mathbb{P}^3 \dashrightarrow \mathbb{P}^3$$
$$J[x_0:x_1:x_2:x_3]=[x_1x_2x_3:x_0x_2x_3:x_0x_1x_3:x_0x_1x_2]$$ 
and $L\in Aut(\mathbb{P}^3)$ given by the matrix
\begin{equation*}
L=
\begin{bmatrix}
1 & 0 & 0 &1  \\
0 & 0 & 1 & 1  \\
0 & 1 & 0 & 1  \\
0 & 1 & 1 &0  \\
\end{bmatrix}
\end{equation*}
We consider the birational map $f=L\circ J: \mathbb{P}^k \rightarrow \mathbb{P}^k $ 
$$ f[x_0:x_1:x_2:x_3]=[(x_0+x_3)x_1x_2:(x_2+x_3)x_0x_1:(x_1+x_3)x_0x_2:(x_1+x_2)x_0x_3]$$
Notice that $f$ has four exceptional hypersurfaces: \{$\Sigma_0,\Sigma_1,\Sigma_2,\Sigma_3$\} where $\Sigma_i:=\{x_i=0\}$ with the following orbit data:
$$ \Sigma_0 \rightarrow e_0:=[1:0:0:0] \rightsquigarrow \Sigma_{\beta}$$ 
$$ \Sigma_1 \rightarrow e_{23}:=[0:0:1:1] \rightsquigarrow l_1 \rightarrow e_{23}$$
$$ \Sigma_2 \rightarrow e_{13}:=[0:1:0:1] \rightsquigarrow l_2 \rightarrow e_{13}$$ 
$$ \Sigma_3 \rightarrow [1:1:1:0]\circlearrowleft$$ 
where $"\rightsquigarrow"$ indicates the total transform under $f$ and $\Sigma_{\beta}=\{[x_0:x_1:x_2:x_3]\in \mathbb{P}^3: 2x_0-x_1-x_2+x_3=0\}$, $l_1\subset \Sigma_1$ is the line passing through $e_0$ and $e_{23}$ and $l_2\subset \Sigma_2$ is the line passing through $e_0$ and $e_{13}$.\\ We define the complex manifold $X$ to be $\mathbb{P}^3$ blown up at $e_0,e_{23}\ \text{and}\ e_{13}$ successively. We denote the exceptional fibers on $e_0, e_{23}\ \text{and}\ e_{1,3}$ by $E_0, E_{23}\ \text{and}\ E_{13}$ respectively and the induced map by $f_X:X\dashrightarrow X$. Then we have the following orbit data
$$\Sigma_0 \rightarrow E_0 \rightarrow \Sigma_{\beta} $$
$$\Sigma_1 \rightarrow E_{23} \rightarrow l_1 \rightarrow \gamma \rightarrow l_1$$
$$\Sigma_2 \rightarrow E_{13} \rightarrow l_2 \rightarrow \sigma \rightarrow \l_2$$
$$ \Sigma_3 \rightarrow [1:1:1:0]\circlearrowleft$$ 
where $\gamma \subset E_{23}$ and $\sigma \subset E_{13}$ are lines which are regular.
It follows that no exceptional hypersurface of $f_X$ is mapped into indeterminacy locus $I_{f_X}$, thus $f_X$ is 1-regular. Now, $\langle H_X,E_0, E_{23},E_{13} \rangle$ forms a basis for $H^{1,1}(X)$ and the the action $f_X^*:H^{1,1}(X)\to H^{1,1}(X)$ with respect to this ordered basis is given by the integer coefficient matrix 
\begin{equation*}
f_X^*=
\begin{bmatrix}
3 & 1 & 1 &1  \\
-2 & 0 & -1 & -1  \\
-1 & -1 & -1 & 0  \\
-1 & -1 & 0 & -1  \\
\end{bmatrix}
\end{equation*}
the charactersitic polynomial of $f_X^*$ is $\chi(x)=x^4-x^3-3x^2+x+2$ and the first dynamical degree of $f,$ the largest root of $\chi(x),$  $\lambda_1(f)=2$ is a simple eigenvalue.\\ \indent 
Let $\widetilde{\Sigma_1} \subset X$ denote the strict transform of $\Sigma_1$. Since $e_0, e_{23} \in \Sigma_1 \subset \mathbb{P}^3$, the class $\alpha:=\{\widetilde{\Sigma_1}\}= H_X-E_0-E_{23}.$ Moreover, $\{\widetilde{\Sigma_1}\}$ is nef in codimension 1. Indeed, for any hyperplane $H\subset \mathbb{P}^3$ containing $e_0$ and $e_{23}$ which does not contain $e_{13}$, $\widetilde{H}$ cohomologous to $\widetilde{\Sigma_1}$ but this is a 1-parameter family of hyperplanes and $[\widetilde{H}] \in \alpha$ defines a positive closed $(1,1)$ current. Thus, we infer that $\nu(\Tmin,x)=0$ for every $x\not\in l_1.$ Hence, by Proposition \ref{Bou} $E_{nn}(\{\widetilde{\Sigma_1}\})\subset l_1.$ Now, since $\widetilde{\Sigma_1} \cdot l_1 =-1$ by Proposition \ref{negative} $l_1 \subset E_{nn}.$ Therefore, $E_{nn}(\{\widetilde{\Sigma_1}\})=l_1.$ On the other hand,
$$f^*(\alpha)=f_X^*(H_X-E_0-E_{23})=H_X-E_0+E_{23}=\{\widetilde{H_0}\}+E_{23}$$
 where $H_0$ is any hyperplane in $\mathbb{P}^3$ containing $e_0$ which does not contain $e_{23}$ nor $e_{13}$. Since this is a 2-parameter family by the same argument above we see that $E_{nn}(H_X-E_0+E_{23}) \subset E_{23}.$ Moreover, for a generic line $\sigma \subset E_{23}$ we have $ E_{23}\cdot \sigma=~-1.$ Therefore, $E_{nn}(H_X-E_0+E_{23})= E_{23}$ hence $f^*(\alpha) \not\in \mathcal{E}_1.$  \\

\section{Green Currents}

\begin{proof}[Proof of Theorem \ref{B}]
We fix a smooth representative $\theta \in \alpha$ and let
\begin{equation}
\Tmin= \theta+ dd^c \vmin \label{a}
\end{equation}
denote the current of minimal singularities. Since $\alpha$  is invariant by $dd^c$-lemma \cite[p 149]{GH}, we can write
\begin{equation}
\fraction{1}{\lambda}f^*\Tmin=\theta+ dd^c\phi_1 \label{b}
\end{equation}
where $\phi_1$ is a qpsh function thus, we can assume that $\phi_1 \leq 0.$ Then by definition of $\vmin$ we have $\phi_1 \leq \vmin.$ \newline
Now, by using invariance of $\alpha$ again we write
\begin{equation}
\fraction{1}{\lambda^2}(f^2)^*\Tmin=\theta+ dd^c\phi_2 \label{c}
\end{equation}
since $f$ is 1-regular  by (\ref{b}) we obtain
\begin{equation}
\fraction{1}{\lambda^2}(f^2)^*\Tmin=\fraction{1}{\lambda}f^*(\fraction{1}{\lambda}f^*\Tmin)= \\
\fraction{1}{\lambda}f^*(\theta)+\fraction{1}{\lambda} dd^c\phi_1\circ f \label{d} 
\end{equation}
and using (\ref{a}) and (\ref{b}) we have
\begin{equation}
\fraction{1}{\lambda}f^*(\theta)=\theta + dd^c(\phi_1-\fraction{1}{\lambda}\vmin \circ f)  \label{e}
\end{equation}  
and substituting (\ref{e}) in (\ref{d}) we obtain 
\begin{equation*}
\fraction{1}{\lambda^2}(f^2)^*\Tmin=\theta+dd^c(\phi_1+\fraction{1}{\lambda}(\phi_1-\vmin) \circ f)
\end{equation*}
therefore, by adding a constant if necessary we can choose 
\begin{equation*}
\phi_2= \phi_1+\fraction{1}{\lambda}(\phi_1-\vmin) \circ f
\end{equation*}
since $\phi_1 \leq \vmin$ we get $\phi_2 \leq \phi_1.$ \newline
Iterating this argument we obtain
\begin{equation*}
\fraction{1}{\lambda^n}(f^n)^*\Tmin=\theta+ dd^c\phi_n 
\end{equation*}
where
\begin{equation*}
\phi_n= \phi_1+ \sum_{j=1}^{n-1}\fraction{1}{\lambda^j}(\phi_1-\vmin)\circ f^j
\end{equation*}
for $n \geq 2$ and $\{\phi_n\}$ is a decreasing sequence of negative qpsh functions. Thus, by Hartogs lemma either $\phi_n$ converges uniformly to $-\infty$ or $\phi_n$ converges to some qpsh function $g.$ We will show that the former case is not possible by using a trick due to Sibony \cite{Si}: \newline
Let $R\in \alpha$ be a positive closed (1,1) current. We consider the Cesaro means of the form
\begin{equation*}
R_N=\fraction{1}{N}\sum_{i=0}^{N-1}\fraction{1}{\lambda^i}(f^i)^*R
\end{equation*}
notice that $R_N$'s are positive closed (1,1) currents and   $\lVert R_N\lVert = \lVert R\lVert $ where $\lVert R \lVert = \int_X R \wedge w^{k-1}$.
Therefore we can extract a subsequence $R_{N_k}\rightarrow S$ for some positive closed current $S\in \alpha$ such that $ f^*S=\lambda S$ and we have
\begin{equation*}
S=\theta+dd^c u
\end{equation*}
where $u$ is a qpsh function. Then, by invariance of S we get 
\begin{equation*}
\fraction{1}{\lambda}(f^*\theta+ dd^c u\circ f)= \theta+dd^c u \label{f}
\end{equation*}  
and by (\ref{e}) we infer
\begin{equation}
\theta+dd^c(\phi_1-\fraction{1}{\lambda}(\vmin+u)\circ f)=\theta+ dd^c u.
\end{equation}
Thus, by adding a constant to $u$ we can assume that 
\begin{equation}
\phi_1-\fraction{1}{\lambda}\vmin \circ f = u-\fraction{1}{\lambda}u \circ f \label{g}
\end{equation}
Pulling back (\ref{g}) by $\fraction{1}{\lambda}f$ and adding the result to (\ref{g}) again we obtain
\begin{equation*}
\phi_n-\fraction{1}{\lambda^n}\vmin \circ f^n =u-\fraction{1}{\lambda^n}u \circ f^n 
\end{equation*}
Since $u$ is qpsh it's bounded from above $u \leq C$ and we have,
\begin{equation*}
\phi_n \geq u + \fraction{1}{\lambda^n}\vmin \circ f^n-\fraction{1}{\lambda^n}C.
\end{equation*}
Thus, we infer that $\phi_n$ converges to $g_{\theta}$ for some qpsh function $g_{\theta}.$  We denote the limit current by $T_{\alpha}=\theta+dd^c g_{\theta}.$ Since it's a limit of positive closed currents belonging to ${\alpha}$, $T_{\alpha} \in \alpha$ is a positive closed current. Moreover, by continuity  of $f^*$ we have $f^*T_{\alpha}=\lambda T_{\alpha}.$ 
Now, we will show that $g_{\theta}$ depends only on the class $\alpha$: \newline 
First of all, since $\fraction{1}{\lambda^{n}} \vmin \circ  f^{n} \rightarrow  0 \ in \ L^1(X)$ we get 
\begin{equation*}
T_{\alpha}=\lim_{n\rightarrow \infty}\fraction{1}{\lambda^{n}} (f^n)^*\Tmin=\lim_{n\rightarrow \infty}\fraction{1}{\lambda^{n}} (f^n)^*\theta
\end{equation*}
Now, if $\theta'$ is another smooth form representing the class $\alpha$ then by $dd^c$-lemma we can write $\theta'=\theta+dd^c\varphi$ where $\varphi$ is a smooth function and since X is compact $\varphi$ is bounded. Therefore, the current $T_{\alpha}=\theta+dd^c g_{\alpha} $ is independent of the choice of the representative form. So far, we have proved the first part. \newline
To prove (1): let $\sigma \in \alpha$ be an invariant current i.e. $f^*\sigma=\lambda \sigma$ and $\sigma=\theta+dd^c \psi$ for some qpsh function $\psi \leq 0$. Then by the same argument as above we obtain
\begin{equation*}
\phi_n-\fraction{1}{\lambda^n}\vmin \circ f^n =\psi-\fraction{1}{\lambda^n}\psi \circ f^n + C\sum_{j=0}^{n-1}\fraction{1}{\lambda^j}
\end{equation*}
thus,  $g_{\alpha} +C_1 \geq \psi.$ \newline
Proof of (2) appears in the literature, see \cite{Si}, \cite{G04}. 
  
\end{proof}

\begin{rem}
 Notice that in the proof of Theorem 4.1 to get the convergence $\lim_{n\rightarrow \infty}\fraction{1}{\lambda^{n}} (f^n)^*\Tmin=T_{\alpha}$ we only need that $\{\fraction{1}{\lambda^{n}}  \vmin \circ f^n\}$ is locally bounded near some point $x \in X.$ 
 \end{rem}
 Theorem \ref{B} was proved by Fornaess and Sibony when $X=\mathbb{P}^k$ (see \cite{FS} and \cite{Si}). More recently, it was proved in \cite{G04} under a cohomological assumption which we replace here by a weaker dynamical assumption. See also \cite{FG},\cite{DF}, \cite{G02},\cite{DG},\cite{DDG} for similar constructions and \cite{N} for the case of non 1-regular meromorphic self maps of $\Bbb{P}^k$. 
 \\  \indent
It seems that the first two assumptions (1-regularity and $\lambda_1>1$) in Theorem \ref{B} are quite natural. Indeed, if $\lambda_1(f)=1$ then it follows from concavity of the function $j \rightarrow \log(\lambda_j)$ and the upper bound for the entropy \cite{DS04} that $h_{top}(f)=0.$ \newline \indent
What about the condition $(\star)$? If $\alpha$ is a K\"ahler class then $\vmin \equiv 0$ thus, $(\star)$ holds. More generally, if $\alpha$ can be represented by a semi-positive form then $\vmin$ is bounded  hence $(\star)$ holds. This is the case for complex homogenous manifolds (i.e. when the group of automorphisms Aut(X) acts transitively on X). Indeed, if X is a complex homogeneous manifold, then every positive closed current T can be approximated by positive smooth forms $\theta_{ \epsilon } \in \{T\}$ (see \cite{H94}). Therefore any psef class can be represented by a semi-positive form. If $X$ is a compact K\"ahler surface and $\lambda_1(f)^2>\lambda_2(f)$, it is well known that $\lambda_1$ is a simple eigenvalue (see \cite{DF}). Furthermore, the corresponding eigenvector $\alpha~ \in \psef$ can be represented by a positive closed $(1,1)$ current with bounded potentials (see \cite{DDG}). Thus, in this case $(\star)$ holds.  \\
\indent However, there are some examples for which the condition $(\star)$ does not hold:
\begin{example}
Let $f:\mathbb{P}^2\rightarrow \mathbb{P}^2$
$$f[x_0:x_1:x_2]=[x_0^2:x_1^2:x_2^2].$$ 
Then $f$ is a holomorphic map with the totally invariant point $p=[1:0:0]$ i.e. $f^{-1}(p)=p.$  Let $\pi:X\rightarrow \mathbb{P}^2$ denote the blow-up of $\mathbb{P}^2$ at $p.$ We let $f_X$ denote the induced map and $E:=\pi^{-1}(p)$ denote the exceptional fiber. Then $f^*_X\{[E]\}=2\{[E]\}$ where $[E]$ denotes the current of integration along $E.$ Notice that $\alpha:=\{[E]\}$ contains only one positive closed $(1,1)$ current namely, $[E].$  Thus, $\Tmin=[E]$ and $E=\{\vmin=-\infty\}.$ We will show that 
$$Vol(\frac{1}{2^n}\vmin \circ f_X^n<-1)\not\rightarrow 0. $$
Indeed, we choose the local coordinates $(s,\eta)$ on $X$ such that $\pi(s,\eta)=[1:s:s\eta]$ Then in these coordinates $E=\{s=0\},$ $\vmin(s,\eta)=\log|s|$ and
$$\frac{1}{2^n}\vmin \circ f_X^n=~\log|s|$$
thus,
$$\{|s|<e^{-1}\}\subset \{\frac{1}{2^n}\vmin \circ f_X^n<-1\}$$
for every $n\in \mathbb{N}$ and the claim follows.
\end{example}
 
 \begin{prop}\label{A}
 Let $f:X\dashrightarrow X$ be a dominant 1-regular meromorphic map and $f^*\alpha=\lambda \alpha$ for some $\alpha \in \psef$ with $\lambda>1.$ Assume that there exists a positive closed $(1,1)$ current $T:=\theta+dd^c\phi$ such that $$T=\lim_{n\rightarrow \infty}\frac{1}{\lambda^n}(f^{n})^*\theta$$ 
 for some (equivalently for every) smooth form $\theta \in \alpha.$ Then $(\star)$ holds.
 \end{prop}
\begin{proof}
Since $\theta\in \alpha$ is a smooth form we have $T=\theta+dd^c\phi$ where $\phi$ is a qpsh function and $\phi\leq \vmin + O(1).$ It's enough to show that $\frac{1}{\lambda^n}\phi \circ f^n \rightarrow 0$ in $L^1(X).$  Let
$$\frac{1}{\lambda}f^*\theta=\theta+dd^c\gamma$$
for some $\gamma \in L^1(X)$. By continuity of $f^*$ we have $f^*T=\lambda T.$ Thus, by adding a constant to $\gamma$ if necessary we may assume that 
$$\gamma+\frac{1}{\lambda}\phi\circ f=\phi$$
then we have 
$$\phi=\sum_{i=0}^{n-1}\frac{1}{\lambda^i}\gamma\circ f^i+ \frac{1}{\lambda^n}\phi\circ f^n$$
and by assumption $\sum_{i=0}^{n-1}\frac{1}{\lambda^i}\gamma\circ f^i\rightarrow \phi$ in $L^1(X)$ hence, $\frac{1}{\lambda^n}\phi \circ f^n \rightarrow 0$ in $L^1(X).$ 
 
\end{proof}

Thus, it follows from the following theorem that if $\lambda_1(f)>1$ is simple and $\alpha$ is nef then condition $(\star)$ holds.
\begin{thm}\cite{DG} \label{DG}
Let $f:X\dashrightarrow X$ be a dominant meromorphic map. Assume that $f$ is 1-regular and $\lambda:=\lambda_1(f) > 1$ is the unique simple eigenvalue and  $\alpha \in H^{1,1}(X,\mathbb{R})$ is the corresponding eigenvector. If $\alpha$ is nef then for every smooth form $\theta \in \alpha$ we have the limit $T_{\alpha}:=\lim_{n\rightarrow \infty} \frac{1}{\lambda^n}(f^n)^*(\theta)$ which depends only on the class $\alpha$. Moreover $T_{\alpha}$ is a positive closed $(1,1)$ current satisfying $f^*T_{\alpha}=\lambda T_{\alpha}.$
 \end{thm}
 \section{An Algebraic Criterion}
The following result is a consequence Hodge index theorem \cite[Lemma 3.39]{KM}:
\begin{lem}[Negativity Lemma]
Let $\pi:Z \dashrightarrow Y$ be a proper birational morphism between normal projective varieties $Z$ and $Y$. Let $-E$ be a $\pi$-nef $\mathbb{R}$-divisor on $Z$ (that is $(-E) \cdot C\geq 0$ for every $\pi$-exceptional curve $C$). Then
\begin{itemize} 
\item[(1)] $E$ is effective if and only if $\pi_*E$ is effective
\item[(2)] Assume that $E$ is effective. Then for every $y\in Y$ either $\pi^{-1}(y) \subset E$ or $supp(E) \cap \pi^{-1}(y)=\emptyset.$
\end{itemize}
\end{lem}
 \begin{prop} \label{p}
Let $X$ be a projective manifold and $f:X\dashrightarrow X$ be a dominant rational map and $\omega$ be a K\"ahler form. If $p\in I_f$ then $\nu(f^*\omega,p)>0.$
\end{prop}
\begin{proof}
 We consider the pull-backs
\begin{equation} \label{l}
(\pi_1)^*f^*\omega=(\pi_1)^*(\pi_1)_*(\pi_2)^*\omega=(\pi_2)^*\omega+ E  
\end{equation}
where $E$ is a (possibly trivial) $\pi_1$-exceptional divisor. We claim that $E$ is a non-trivial effective divisor. Indeed, for any $\pi_1$-exceptional curve $C$ we have 
$$(\pi_1)^*f^*\omega \cdot C=0$$
and since $\omega$ is a K\"ahler form we get
$$0\leq \langle \omega, \pi_2(C)\rangle=(-E) \cdot C$$
Thus, by Negativity lemma we conclude that $E$ is effective.\\ \indent
 Let us fix $p\in I_f.$ Since $\dim(f(p))\geq 1,$ there exists a curve $C\subset \pi^{-1}_1(p)$ such that $C\not \subset \mathcal{E}(\pi_2)$ and by above argument $ E \cdot C<0$ hence $E$ is non-trivial. Moreover, $\pi_1^{-1}(p)\subset E$. \\ \indent
 Since the left hand side of (\ref{l}) defines a positive closed $(1,1)$ current we infer that $\nu(\pi_1^*f^*\omega,q)=mult_q(E)$ for any $q\in \pi_1^{-1}(p)$ and
 $$0<\nu((\pi_1)^*f^*\omega,q)$$  
 Now, it follows from Theorem \ref{F} that
 $$0<\nu((\pi_1)^*f^*\omega,q) \leq C\ \nu(f^*\omega,p)$$
 where $C>0$ is a constant which does not depend on $f^*\omega$. 
 \end{proof}
The following result is well-known when $X=\Bbb{P}^k$ \cite{FS} or $X$ is a compact K\"ahler surface \cite{DF}. To our knowledge it is new in this generality:
\begin{thm}\label{t}
Let $X$ be a projective manifold and $f:X\dashrightarrow X$ be a dominant rational map. Then the following are equivalent:
\begin{itemize}
\item[(i)] $(f^n)^*T=(f^*)^nT$ for every $T \in \mathcal{T}(X)$ and $n=1,2,\dots$
\item[(ii)] $(f^*)^n\omega=(f^n)^*\omega$ for every K\"ahler form $\omega$ on $X$ and $n=1,2,\dots$
\item[(iii)] f is 1-regular
\item[(iv)] There is no exceptional hypersurface $H$ and $n\in \mathbb{N}$ such that $f^n(H-I_{f^n})\subset~I_f$.

\end{itemize}
\end{thm}
\begin{proof}
$(i) \implies (ii) \implies (iii)$ is clear.\\
$(iii) \implies (iv)$ Let be $\omega$ be a K\"ahler form. Assume that there exists an exceptional hypersurface $H$ such that $f^n(H-I_f)\subset I_f$ for some $n$. By replacing $f$ with $f^n$ we may assume that $f(H-I_f)\subset I_f.$ Then by Proposition \ref{p} for each $p\in f(V-I_f)\subset I_f$ we have $\nu(f^*\omega ,p)>0$ but this implies that $\nu ((f^*)^2\omega, q)>0$ for every $q\in H-I_f$. However, $(f^2)^*\omega$ is an $L^1_{loc}$ coefficient form and  does not charge $H$ and hence the cohomology classes of $(f^*)^2\omega$ and $(f^2)^*\omega$ are different.\\
$(iv)\implies (i)$ Let $T\in \mathcal{T}(X).$ Notice that $f^{n-1}$ and $f^n$ are both holomorphic on $X-(I_f \cup f^{-1}(I_f))\dots \cup f^{n-1}(I_f))$ Thus, $(f^n)^*T=(f^*)(f^{n-1})^*T$ on this set. Since there is no hypersurface contained in $X-(I_f \cup f^{-1}(I_f))\dots \cup f^{n-1}(I_f))$ we get the equality on $X.$
\end{proof}

\begin{lem} \label{lem}
Let $\pi:Z\rightarrow Y$ be a proper modification between smooth projective varieties. Let $\eta$ be a smooth closed real $(1,1)$ form on $Z$ such that $\langle \eta,C \rangle \geq 0$ for every $\pi$-exceptional curve $C$. Then $\pi_*\eta$ has potentials bounded from above.
\end{lem}
\begin{proof}
Let $\pi_*\eta=\eta'+dd^cu$ for some smooth form $\eta'$ and  $u\in L^1(X)$. We claim that $u$ is bounded from above. Indeed,
$$\pi^*\pi_*\eta=\pi^*\eta'+dd^c(u\circ \pi)=\eta +E$$
where $E$ is an $\mathbb{R}$ divisor supported in $\mathcal{E}(\pi).$ Since
$$0\leq \langle \eta,C\rangle=\langle-E,C\rangle$$
for every $\pi$-exceptional curve $C$ by negativity lemma $E$ is an effective divisor. Hence, $u\circ \pi$ is qpsh on $Z$ and bounded from above. Thus, so is $u$.
\end{proof}
\begin{prop} \label{bdd}
Let $X$ be a projective manifold and $f:X\dashrightarrow X$ be a dominant rational map. Let $\theta$ be a smooth closed real $(1,1)$ form on $X$ such that $\langle\theta,C\rangle\geq 0$ for every curve $C\subset E^-_f:=f(I_f)$ then the potentials of $f^*\theta$ are bounded from above.
\end{prop}
\begin{proof}
We write $f^*\theta=(\pi_1)_*(\pi_2)^*\theta$ where $(\pi_2)^*\theta$ is a smooth form on the desingularization of the graph of $f$,  $\tilde{\Gamma} \subset X\times X.$ Notice that for any $\pi_1$-exceptional curve $C\subset \tilde{\Gamma}$, $\pi_2(C)$ is either a point in $X$ or a curve in $E^-_f.$ Thus, we have 
$$\langle \pi_2^*\theta,C\rangle = \langle \theta,\pi_2(C)\rangle \geq 0$$
Then, applying Lemma \ref{lem} with $\eta=(\pi_2)^*\theta$ the assertion follows. 
\end{proof}
 For a convex cone $\mathcal{C}$ in a finite dimensional vector space $V$ we define the dual cone $\mathcal{C}^{\vee}$ to be the set of linear forms in $V^*$ which have non-negative values on $\mathcal{C}$. Moreover, by Hahn-Banach thoerem we have $\mathcal{C}^{\vee\vee}=\overline{\mathcal{C}}.$
\begin{thm} \label{green}
Let $X$ be a projective manifold and $f:X\dashrightarrow X$ be a 1-regular dominant rational map. We assume that $\lambda_1(f)>1$ is a simple eigenvalue of $f^*$ and $\alpha$ denote the corresponding eigenvector. If $\alpha\cdot C\geq0$ for every curve $C\subset E^-_f$ then for every smooth representative $\theta \in \alpha$ we have   
$$T_{\alpha}= \lim_{n\rightarrow \infty} \frac{1}{\lambda^n}(f^n)^*\theta$$
exists. Moreover, $T_{\alpha}$ is a positive closed $(1,1)$ current such that $f^*T_{\alpha}=\lambda T_{\alpha}$.
\end{thm}
\begin{proof}
We will sketch the proof:
Let $\theta\in \alpha$ be a smooth representative then by $dd^c$-lemma we have 
$$\frac{1}{\lambda}f^*\theta=\theta+dd^c\gamma$$
and by  Proposition \ref{bdd}, $\gamma \in L^1(X)$ is bounded from above. Thus, we may assume that $\gamma\leq 0.$ Iterating this equation we get 
$$\frac{1}{\lambda^n}(f^n)^*\theta=\theta+dd^c\gamma_n$$
where
$$\gamma_n=\sum_{j=0}^{n-1} \frac{1}{\lambda^j}\gamma \circ f^j.$$
Then $\{\gamma_n\}$ is a decreasing sequence in $L^1(X).$ It follows from Sibony's argument \cite{Si} that  $\{\gamma_n\}\geq \phi$ for some qpsh function $\phi.$ Thus, $\gamma_n\rightarrow \gamma_{\infty}$ for some $ \gamma_{\infty}\in L^1(X).$ Therefore,
$$T_{\alpha}:=\theta+dd^c \gamma_{\infty}$$
defines a closed $(1,1)$ current. It follows from continuity of $f^*$ that $f^*T_{\alpha}=\lambda T_{\alpha}.$\\ \indent
 It remains to show that $T_{\alpha}$ is positive. We will follow the arguments from \cite{BS3} and \cite{DG}. It is enough to show that for every smooth cutoff function $\chi$ supported in a coordinate chart $U \subset X$ and positive $(k-1,k-1)$ form $\sigma$ which is constant relative the coordinates on $U$
 $$\langle T, \chi \sigma \rangle \geq 0.$$ 
 Since $\lambda$ is simple it follows form [\cite{BS3}, Lemma 1.3] that the sequence $\{\frac{1}{\lambda^n}(f^n)_*(\chi \sigma)\}$ has weak limit points which are positive and closed. Moreover, since $(f_*)|_{H^{k-1,k-1}(X,\mathbb{R})}$ preserves classes; these limit points belongs to the dual of the psef cone. Thus,
 $$\langle T_{\alpha},\chi \sigma\rangle=\lim_{n_k \rightarrow \infty} \langle\frac{1}{\lambda^{n_k}}(f^{n_k})^*\theta,\chi \sigma \rangle=\langle \theta ,S\rangle \geq 0$$ 
 where $S=\lim_{n_k\rightarrow \infty}\ \frac{1}{\lambda^{n_k}}(f^{n_k})_*(\chi \sigma).$
\end{proof}

Note that if $\alpha$ is nef then $ \alpha \cdot C \geq 0$ for every curve $C.$ This is the case when $X$ is a compact K\"ahler surface and $\lambda_1(f)> \lambda_2(f)$ and the corresponding results were obtained in \cite{DDG} as a consequence of so called "push-pull formula" (\cite{DF}).\\ \indent
  If there exists an irreducible curve $C\subset E^-_f$ such that $\alpha \cdot C<0$ then $C\subset E_{nn}(\alpha).$ Thus, if $\dim(E_{nn}(\alpha)\cap E^-_f)=0$ then $\alpha \cdot C \geq 0$ for every curve $C \subset E^-_f$ .

The following result follows from Proposition \ref{A}
\begin{cor} \label{greens}
Let $f:X\dashrightarrow X$ and $\alpha$ be as in Theorem \ref{green} then 
$$\frac{1}{\lambda^n}\vmin \circ f^n \rightarrow 0 \ in\ L^1(X).$$
\end{cor}

 \section{Noetherian Mappings}
 Let $\mathbb{P}^d$ denote the complex projective space of dimension $d$ and for a point $x \in \mathbb{P}^d$
 $$x=[x_0:x_1:\dots:~x_d]$$
 denotes the homogenous coordinates on $\mathbb{P}^d$. For a subset $I \subset \{0,1,..,d\}$ we denote its complement by $\hat{I}:=\{0,1,..,d\}-I$ and its cardinality by $|I|$. We also define the sets
 $$\mathcal{D}_I:=\{[x_0:\dots:x_d]\in \mathbb{P}^d: x_{i_1}=x_{i_2}\ \text{for every}\ i_1, i_2 \in \hat{I}\}.$$
 In particular, if $I=\{i\}$ then we set  $\mathcal{D}_i:=\mathcal{D}_{\{i\}}$ which is a complex line. We also denote
 $$\Sigma_I:=\{[x_0:\dots:x_d]\in \mathbb{P}^d: x_i=0\ \text{for}\ i \in I\}.$$\\ 
 In this section, we consider the maps of the form  $f=L \circ J:\mathbb{P}^d \dashrightarrow \mathbb{P}^d$ where $J:\mathbb{P}^d \dashrightarrow \mathbb{P}^d$ is the involution defined by   
 $$J[x_0:x_1:\dots:x_d]=[x_0^{-1}:x_1^{-1}:\dots:x_d^{-1}]=[x_{\hat{0}}:x_{\hat{1}}:\dots:~x_{\hat{d}}] $$
 with $ x_{\hat{\jmath}}:=\displaystyle\prod^d_{\substack{i=1 \\
i\not=j}}x_i$ and $L$ is a linear map given by $(d+1)\times (d+1)$ matrix of the form
 \begin{equation}
 L=
 \begin{bmatrix}
a_0-1 & a_1 & a_2 & \ldots & a_d
 \\
a_0 & a_1-1 & a_2 & \ldots & a_d
 \\
a_0 & a_1 & a_2-1 & \ldots & a_d
\\
\vdots & \vdots & \vdots & \ddots & \vdots \\

a_0 & a_1 & a_2 & \ldots & a_d-1
\end{bmatrix}
\end{equation}
 with $a_j \in \mathbb{C}$ and $\sum_{j=0}^d a_j=2.$ It follows that $\det (L)=(-1)^{d}$ and $L$ is involutive that is $L = L^{-1}$ in $PGL(d+1,\mathbb{C})$. A map of this form is called Noetherian mapping in \cite{BHM}.
 Notice that $f$ is a birational mapping with $f^{-1}=J \circ L.$ Moreover, the indeterminacy locus is given by
 $$\displaystyle I_f=\bigcup_{|I|\geq 2} \Sigma_I.$$ \indent
 For a point $p \in \mathbb{P}^d,$ we define it's orbit $\mathcal{O}(p)$ as follows: $\mathcal{O}(p)=\{p\}$ if $p \in I_f$ and $\mathcal{O}(p) =\{p,f(p),f^2(p),..,f^{N-1}(p)\}$ if $f^j(p)\not\in I_f$ for $0\leq j\leq N-2$ and $f^{N-1}(p) \in I_f$ for some $N \in \mathbb{N}$, otherwise $\mathcal{O}(p)=\{p,f(p),f^2(p),...\}.$ If $\mathcal{O}(p)$ is finite with $f^{N-1}(p) \in I_f$ then we say that $p$ has a singular orbit of length $N$; otherwise we say that it has a non-singular orbit.\newline \indent
 A hypersurface $H$ is called exceptional if  $\dim f(H-I_f)<d-1.$ The only exceptional hypersurfaces of $f$ are of the form $$\Sigma_i:=\{[x_0:\dots:x_d]\in \mathbb{P}^d: x_i=0\}.$$ 
 In fact $p_i:=f(\Sigma_i-I_f)$ is the $i^{th}$ column of the matrix $L$. It follows from Theorem \ref{t} that $f$ is 1-regular if and only if $f^n(\Sigma_i-I_{f^n})\not\subset I_f$ for every $n\geq1$ and $i \in \{0,1,\dots,d\}.$ We denote the orbit of $\Sigma_i$ by $\mathcal{O}_i:=\mathcal{O}(p_i).$ Then it is easy to see that the orbit $\mathcal{O}_i$ is given by $p_{i,j}=[1:\dots:1:\frac{j(a_i-1)}{ja_i-(j-1)}:1:\dots :1]$  for $j=1,2\dots$ and $\mathcal{O}_i$ is contained in the complex line $\mathcal{D}_i.$ Thus, $\mathcal{O}_i \cap \mathcal{O}_j =\emptyset$ for $i\not=j.$ In particular, $\mathcal{O}_i$ is singular if and only if $a_i=\frac{N-1}{N}$ for some $N \in \mathbb{N}_+$ and in this case $f^{N-1}(p_i)=e_i:=[0:\dots:0:1:0\dots:0]$ where 1 is in the $i^{th}$ component. Let $\mathcal{O}_i=\{p_{i,j}\}_{j=1}^{N_i}$ be a singular orbit then we denote its length by $|\mathcal{O}_i|:=N_i$.  
 We define the set $S:=\{i : \mathcal{O}_{i}\ \text{is singular}\}$ and we also set $\displaystyle \mathcal{O}_S:=~\bigcup_{i\in S} \mathcal{O}_i$. By conjugating $f$ with an involution, without lost of generality we may assume that $S=\{0,1,\dots,k\}$ with $N_0\leq N_1\leq \dots \leq N_k$ where $0\leq k\leq d+1$  and we define $l$ by $l:= \#\{i\in S : a_i= 0\}$ if the later set is non-empty otherwise we set $l=0$.\\ \indent
 Let $\pi:X\rightarrow \mathbb{P}^d$ be the complex manifold obtained by blowing up the points in the set $ \mathcal{O}_S$ successively. Then $f$ induces a birational map $f_X:X \dashrightarrow X.$\\ \indent
  We denote the exceptional fiber over the point $p_{i,j} \in \mathcal{O}_S$ by $P_{i,j}:=\pi^{-1}(p_{i,j}).$ We also define the class  $H_X:=\pi^*H$ where $H\subset \mathbb{P}^d$ is class of a generic hyperplane and let $P_{i,j}$ denote the class of exceptional divisor over $p_{i,j}.$ Then $\{H_X,P_{0,1},P_{0,2},..,P_{k,N_{k}}\}$ forms a basis for $H^{1,1}(X,\mathbb{R})$ and the action of $f_X^*$ on $H^{1,1}(X,\mathbb{R})$ is given by
 \begin{eqnarray} 
 f_X^*(H_X)&=&dH_X-(d-1)\sum_{i \in S}P_{i, N_{i}} 
 \\
  f_X^*(P_{i,j+1})&=&P_{i,j}\ \text{for}  \ 1\leq j \leq  N_{i}-1
  \\
   f_X^*(P_{i, 1})&=&\{\widetilde{\Sigma_{i}}\}=H_X-\sum_{\substack{j\in S \\
j\not=i}}P_{j,N_{j}}
 \end{eqnarray}
 where $\widetilde{\Sigma_{j}} \subset X$ denotes the strict transform of $\Sigma_{j}\subset \mathbb{P}^d$ (see [\cite{BK},\S3] for details). \\ \indent
 
 \begin{thm}\cite{BK}\label{BK} Let $f_X:X\dashrightarrow X$ be as above then $f_X$ is 1-regular and characteristic polynomial of $f_X^*$ is given by
 $$\chi(x)= (x-1)^l\big[(x-(d-l))\prod_{j=l}^{k}(x^{N_j}-1)+(x-1)\sum_{j=l}^{k}\prod^{k}_{\substack{i=l  \\
i\not=j}} (x^{N_i}-1).$$
 Moreover, if $S\not=\emptyset$ and 
 \begin{equation} \label{*}
d-l\geq 3
 \end{equation}
 then $d-l-1\leq\lambda:=\lambda_1(f)\leq d$ is the unique eigenvalue of $f_X^*$ of modulus greater than one and is a simple root of $\chi(x).$
 \end{thm}
    In the sequel, we will assume that $d \geq3$ and $f_X$ is as in Theorem \ref{BK} so that $(\ref{*})$ holds. We denote the corresponding eigenvector by $\alpha_f \in H^{1,1}(X,\mathbb{R})$ with $f^*_X\alpha_f=\lambda \alpha_f$ and we normalize it so that
    $$\alpha_f=H_X-c \cdot E$$
    where $c=(c_{0,1},c_{0,2},\dots,c_{k,N_{k}})$ and $E=(P_{0,1},P_{0,2},\dots,P_{k,N_{k}}).$ 
\begin{lem} \label{lemma} Let $\alpha_f=H_X-c \cdot E$ be as above then for every $0\leq i \leq k$
\begin{enumerate}
\item $c_{i,j+1}=\lambda c_{i,j}$ for $1\leq j \leq N_{i}-1$ 
\item $\sum_{i= 0}^{k}c_{i,1}=d-\lambda$
\item $c_{i,1}=\frac{\lambda-1}{\lambda^{N_{i}}-1}>0$ 
\item $\sum_{j=1}^{N_i}c_{i,j}=1.$
\end{enumerate}
\end{lem} 
\begin{proof}
(1) and (2) follows from the invariance of $\alpha_f$ and (6.2)-(6.4).\\
(3) For fixed $i$ we compare the coefficient of $P_{i, N_{i}}$ on the both sides of $f^*\alpha_f=~\lambda \alpha_f$ and obtain
$$(d-1)-\sum_{\substack{j= 0 \\
j\not=i}}^{k} c_{j, 1}=\lambda c_{i, N_i}$$
then by (1) and (2) we get
$$c_{i,1}=\frac{\lambda-1}{\lambda^{N_i}-1}.$$
(4) follows from (1) and (3).
\end{proof}    
 \begin{prop}\label{1}
 The class $\alpha_f \in H^{1,1}_{nef}(X,\mathbb{R})$ if and only if $|S|\leq1.$
\end{prop}
 \begin{proof}
 If $S=\emptyset$ then $X=\Bbb{P}^d$ and  $\alpha_f=\{\omega_{FS}\}$ which is K\"ahler.\\
 Assume that $|S|=1.$ Then $\mathcal{O}_0$ is singular and the orbit is
 $$\Sigma_0 \rightarrow p_1 \rightarrow \dots \rightarrow p_N=e_0.$$
 Let $H_i\subset \mathbb{P}^d$ denote a hyperplane such that $p_i \in H_i$ and $p_j \notin H_i$ for $j\not=i.$ Notice that $H_i$'s form a $(d-1)$-parameter family of hyperplanes. Since $\{\widetilde{H_i}\}=H_X-P_i$ by Lemma \ref{lemma} we can represent the class $\alpha_f$ as the class of effective divisor $\sum_{i=1}^{N}c_i\widetilde{H_i}$ where $\widetilde{H_i}$ is the the strict transform of $H_i.$ Since $\sum_{i=1}^{N}c_i[\widetilde{H_i}] \in \alpha_f$ defines a positive closed $(1,1)$ current, we infer that $\nu(\Tmin,x)=0$ for every $x \in X.$ Thus, it follows from Proposition \ref{Bou} that $\alpha_f$ is nef. \\
\indent Now, we will prove that if $|S| \geq 2$ then $\alpha_f$ is not nef. Indeed, let $\mathcal{O}_{i_1}$ and $\mathcal{O}_{i_2}$ be two singular orbits then by Lemma \ref{lemma} and Theorem \ref{BK}
 $$c_{i_j,N_{i_j}} > 1-\frac{1}{\lambda}\geq1-\frac{1}{d-l-1}\geq\frac{1}{2}\ \text{for}\ j=1,2. $$
Let $\ell$ denote the complex line passing through the points $e_{i_1}$ and $e_{i_2}$ and $\widetilde{\ell} $ be its strict transform in $X$ then
$$\alpha_f \cdot \widetilde{\ell} =1-c_{i_1,N_{i_1}}-c_{i_2,N_{i_2}} <0$$
hence by Proposition \ref{negative} we get $\widetilde{\ell} \subset E_{nn}(\alpha_f)$.
   \end{proof}
  Let $\Sigma_{I}\subset \mathbb{P}^d$ be as above, we also write $\Sigma_{I}$ for its strict transform inside $X.$
\begin{prop} \label{2} 
If $1 \leq k \leq d-1$ and $2\leq N:=N_{i}$ for every $0\leq i\leq k$ then 
\begin{equation*}
E_{nn}(\alpha_f)=
\begin{cases}
\Sigma_{\{k+1,\dots,d\}} & \text{if } k \leq d-2,\\
\displaystyle \bigcup^k_{i=0}\Sigma_{\{i,d\}} & \text{if } k=d-1
\end{cases}
\end{equation*}
In particular, $1\leq \dim_{\mathbb{C}} E_{nn}(\alpha_f)\leq (d-2)$ and $E_{nn}(\alpha_f) \subset I_{f_X}$ is algebraic.
\end{prop}
\begin{proof}
It follows from Lemma \ref{lemma} that $c_{i_1,l}=c_{i_2,l}$ for all $i_1, i_2 \in S=\{0,\dots, k\}$ and $1 \leq l \leq N.$ We denote $c_l:=c_{i,l}$ for $i \in \{0,\dots,k\}.$\\
If $|S|=2$ then $\Sigma_{\{2,\dots,d\}}$ is a complex  line and in the proof of Proposition \ref{1} we have already showed that $\Sigma_{\{2,\dots,d\}} \subset E_{nn}(\alpha_f).$\\
Assume that $3 \leq|S|=k+1\leq (d-1)$ and let $p \in \Sigma_{\{k+1,\dots,d\}}\cong \mathbb{P}^{k}$ be a point. Let $\gamma \subset \Sigma_{\{k+1,\dots,d\}}$ be an algebraic irreducible curve of degree $k$ such that $p,e_i \in \gamma$ for every $0\leq i\leq k.$ Then by Lemma \ref{lemma} and by (\ref{*}) we have $c_N>1-\frac{1}{d-1}$ and 
$$ \alpha_f \cdot \gamma=k-(k+1) c_N<\frac{k+1}{d-1}-1\leq0.$$ 
Thus, by Proposition \ref{negative} we get $\Sigma_{\{k+1,\dots,d\}} \subset E_{nn}(\alpha_f).$ \\ \indent
If $k=d-1$ we can apply the same argument to $\{0,\dots,d-1\}-\{i\}$ for $0 \leq i \leq d-1 $. \\    \indent  
To prove the reverse inclusion we will represent the class $\alpha_f$ by effective divisors: Notice that each $p_{i,l}=[1:\dots:1:\frac{l(a_i-1)}{la_i-(l-1)}:1:\dots :1]\in\mathcal{D}_i $ which is a complex line. Let $H_{l} \subset \mathbb{P}^d$  be a hyperplane such that $p_{i,l} \in H_{l}$ for every $0 \leq i \leq k$ and $p_{i,m} \not\in H_{l}$ for $m \not=l.$ This is a $d-k-1$ parameter family of hyperplanes for each $l.$ Then, the class $\{\widetilde{H_l}\}=H_X-\sum_{i=0}^{k} P_{i,l}$ where $\widetilde{H_l}$ denotes the strict transform of $H_l.$ Hence, by Lemma \ref{lemma} we can represent $\alpha_f$ by 
\begin{equation} \label{rep1}
\alpha_f=\sum_{l=1}^N c_l\{\widetilde{H_l}\}.
\end{equation} 
\indent
Next, we assume that $k\leq d-2.$ We consider the hyperplanes of the form $D_i=\{x\in  \mathbb{P}^d: 2x_i-x_{d-1}-x_{d}=0\}$ where $0\leq i\leq k$ is fixed. Then the complex line $\mathcal{D}_j \subset D_i$ for $0\leq j\not=i\leq k$  and $\mathcal{O}_i \cap D_i=\emptyset.$ Thus, $\{\widetilde{D_i}\}=H_X-\displaystyle \sum_{\substack{l=1 \\0\leq j\not=i\leq k}}^NP_{j,l}.$ We also denote $H_{\Sigma_{\{k+1,\dots,d\}}} \subset \mathbb{P}^d$ be a hyperplane containing $\Sigma_{\{k+1,\dots,d\}}$ such that $\{H_{\Sigma_{\{k+1,\dots,d\}}}\}=H_X-\sum_{i\in S}P_{i,N}$ (Eg. $H_{\Sigma_{\{k+1,\dots,d\}}}=\Sigma_j$ for some $k+1\leq j \leq d$). Then by Lemma \ref{lemma}
\begin{equation}\label{rep2}
\alpha_f=\sigma\sum_{i=0}^k\{\widetilde{D_i}\}+(1-\sigma(k+1))\{\widetilde{H}_{\Sigma_{\{k+1,\dots,d\}}}\}+\mathcal{E}
\end{equation}  
where $\sigma=1-c_N$ and $\mathcal{E}$ is an effective divisor supported on $\displaystyle \bigcup_{\substack{i\in\{0,\dots,k\} \\
1\leq l\leq N-1}} P_{i,l}.$ Indeed, it follows from Lemma \ref{lemma} that 
$$1-\sigma(k+1) =\frac{d-k-1}{\lambda}>0.$$
On the other hand, we can also represent $\alpha_f$ as follows: let $i_1,i_2 \in S$ then
\begin{equation}\label{rep3}
\alpha_f=\sigma(\{\widetilde{D_{i_1}}\}+\{\widetilde{D_{i_2}}\})+(1-2\sigma)\{\widetilde{H}_{\Sigma_{\{k+1,\dots,d\}}}\}+\mathcal{E}'
\end{equation}
where $\mathcal{E}'$ is an effective divisor supported on $\displaystyle \bigcup_{\substack{ i\in\{0,\dots,k\} \\
1\leq l\leq N}} P_{i,l}$ and $\sigma$ is as above. Since the non-nef locus is contained in the intersection of the supports of the effective divisors in (\ref{rep1}),(\ref{rep2}) and (\ref{rep3}) we conclude that  
$$E_{nn}(\alpha_f) \subset \Sigma_{\{k+1,\dots,d\}}.$$
\indent If $k=d-1$ then we claim that $c_N=\frac{d-1}{d}.$ Indeed, by Lemma \ref{lemma}\ (2) $c_1=\frac{(d-\lambda)}{d}$ and by using Lemma \ref{lemma} (3) we get $\lambda^{N-1}=\frac{d-1}{d-\lambda}.$ Then by Lemma \ref{lemma}(1) we get $c_N=\lambda^{N-1}c_1=\frac{d-1}{d}.$ This implies that we can represent $\alpha_f$ as 
\begin{equation}\label{rep4}
\alpha_f=\frac{1}{d}\sum_{i=0}^{d-1}\{\widetilde{L_i}\}+\mathcal{E}
\end{equation}
where $L_i:=\{x\in \mathbb{P}^d: x_i-x_{d}=0\}$ and $\mathcal{E}$ is an effective divisor supported on $\displaystyle \bigcup_{\substack{0\leq i\leq d-1 \\ 1\leq l\leq N-1}} P_{i,l}.$ Now, for fixed $0\leq j\leq d-1 $ we also have
\begin{equation}\label{rep5} 
\alpha_f=\frac{1}{d} \sum^{d-1}_{\substack{ i=0 \\ i\not=j}}\{\widetilde{L_i}\}+\frac{1}{d}\{\widetilde{F_j}\}+\mathcal{E}'
\end{equation}
where $F_j:=\{x\in\mathbb{P}^d: x_j-ax_{d}=0\}$ is a 1-parameter family of hyperplanes and $\mathcal{E}'$ is an effective divisor supported on $\displaystyle \bigcup_{\substack{0\leq i \leq d-1 \\ 1\leq l\leq N-1}} P_{i,l}.$ 
Hence, by (\ref{rep1}), (\ref{rep4}) and (\ref{rep5}) we get
$$E_{nn}(\alpha_f) \subset \bigcup_{i=0}^{d-1}\Sigma_{\{i,d\}}.$$
\end{proof}
Now, we prove that a generic mapping of the form $f=L \circ J$ fall into framework of Theorem \ref{B}:
 \begin{proof}[Proof of Theorem \ref{B2}] 
 If $|S| \leq 1$ then the assertion follows from Proposition \ref{1} and Theorem~\ref{DG}.\\
We assume that $|S|\geq 2$ and set $S=\{0,\dots,k\}.$ By Theorem \ref{green} and Corollary \ref{greens} it is enough to show that $\alpha_f \cdot C \geq 0$ for every algebraic irreducible curve $C\subset E_{f_X}^-$.  Since $f_X$ is biholomorphic near the exceptional fibers $P_{i,j}$'s, the indeterminacy locus is given by 
$$\displaystyle I_{f_X}=\bigcup_{|I|\geq 2}\widetilde{\Sigma_I}.$$
This implies that $f_X(I_{f_X})\subset \bigcup_{i=0}^d L(\widetilde{\Sigma_i})$ where $L(\Sigma_i)=\{x\in \Bbb{P}^d:a\cdot x-x_i=0\}$ and $a=[a_0:\dots:a_d].$ Since $p_{i,j}=[1:\dots:\frac{j(a_i-1)}{ja_i-(j-1)}:\dots :1]$ we infer that $$ \bigcup_{i=0}^d L(\Sigma_i) \cap \mathcal{O}_S=\{p_{0,1},p_{1,1},\dots, p_{k,1}\}.$$
Then for any algebraic irreducible curve $C\subset E_{f_X}^-$ by Lemma \ref{lemma}
\begin{eqnarray*}
\alpha \cdot C &\geq& \deg C- \sum_{i=0}^k c_{i,1} (mult_{p_{i,1}}C) \\
&\geq& \deg C(1- (d-\lambda)) \\ 
& \geq & 0
\end{eqnarray*}
where the last inequality follows from Theorem \ref{BK}.
 \end{proof}
 
 \begin{thm}
Let $f:\Bbb{P}^3 \dashrightarrow \Bbb{P}^3$ be as above. If $|S|\leq 3$ and $2 \leq N:=N_{i}$ for every $i \in S$ then there exists a birational model $\mu:Y \rightarrow \mathbb{P}^3$ such that $f_Y :Y\dashrightarrow Y$ is a dominant 1-regular map with $\lambda:=\lambda_1(f_Y)$ is the unique simple eigenvalue of modulus greater than 1 with the corresponding normalized eigenvector $\widetilde{\alpha_f} \in H^{1,1}_{nef}(Y,\mathbb{R})$. 
 \end{thm}
\begin{proof}
If $|S| \leq 1$ then the assertion follows from Theorem~\ref{BK} and Proposition \ref{1}.\\
We assume that $S=\{0,1\}$ then we define the complex manifold $Y$ to be $X$ blown up along $E_{nn}(\alpha_f)=\Sigma_{23}$ which is a complex line. We denote the projection by $\mu:Y \rightarrow X$ and the exceptional divisor by $\mathcal{F}:=\mu^{-1}(E_{nn}(\alpha_f))$.\\
We first show that the induced map $f_Y:Y\dashrightarrow Y$ is 1-regular: Notice that the only exceptional hypersurfaces of $f_Y$ are $\widetilde{\Sigma}_i$ for $i\not\in S$ and $\mathcal{F}$. Since $f_Y^n(\widetilde{\Sigma}_i-I_{f^n})\not\subset I_{f_Y}$ for $n \geq 1$ and $i\not\in S$ by Theorem \ref{t} it's enough to check that $f_Y^n(\mathcal{F}) \not\subset I_{f_Y}$ for every $n\in \mathbb{N}.$\\ \indent
 We claim that $\overline{f_Y(\mathcal{F}\backslash I_{f_Y})}=\widetilde{L(\Sigma_S)}.$ Indeed, we write $f_Y$ in the local coordinates: $(\eta_1,\eta_2,s)\in Y$ where $\mathcal{F}=\{s=0\}$ and
 $$\pi_Y:Y\to \Bbb{P}^3$$ 
$$\mu(\eta_1,\eta_2,s)=[1:\eta_1:\eta_2s :s]$$
Then, we may identify 
$$f_Y(\eta_1,\eta_2,0)=\eta_1 [a_2:a_2:a_2-1:a_2]+ \eta_1\eta_2 [a_3:a_3:a_3:a_3-1]$$ 
which proves the claim. Since the points $[a_2:a_2:a_2-1:a_2]$ and $[a_3:a_3:a_3:a_3-1]$ have non-singular orbits we conclude that $f_Y$ is 1-regular. Similarly, one can show that $\overline{f^{-1}_Y(\mathcal{F}\backslash I_{f^{-1}_Y})}=\widetilde{J(span\{[a_0-1:a_0:a_0:a_0],[a_1:a_1-1:a_1:a_1]\})}$ where the later set has codimension 2.\\ \indent
Now, $\{H_Y,P_{0,1},P_{0,2},..,P_{1,N},\mathcal{F}\}$ forms an ordered basis for $H^{1,1}(X,\mathbb{R})$ where $H_Y:=\mu^*(H_X)$ and $P_{i,l}:=\mu^*(P_{i,l})$ for each $i,$ $1\leq l\leq N$ and the action of $f_Y^*:H^{1,1}(Y)\rightarrow H^{1,1}(Y)$ is given by 
\begin{eqnarray*}
 f_Y^*(H_Y)&=&3H_X-2P_{0,N}-2P_{1,N}-\mathcal{F} 
 \\
  f_Y^*(P_{i,l+1})&=&P_{i,l}\ \text{for}  \ 1\leq l \leq  N-1\ \text{and}\ i\in S
  \\
   f_Y^*(P_{i,1})&=&\widetilde{\Sigma_{i}},\ \text{for}\  i\in S
   \\
   f_Y^*(\mathcal{F})&=&0
 \end{eqnarray*}
 where $\widetilde{\Sigma_{i}}\subset Y$ denotes the strict transform of $\Sigma_i \subset \mathbb{P}^3.$ Thus, the characteristic polynomial of $f_Y^*$ is given by $p(x)=x\chi(x)$ where $\chi(x)$ is as in Theorem \ref{BK}. This implies that $\lambda=\lambda_1(f_Y)$ is a simple eigenvalue. Moreover, corresponding eigenvector $\widetilde{\alpha_f}$ is of the form 
 $$\widetilde{\alpha_f}=H_Y-c\cdot E-\frac{1}{\lambda}\mathcal{F}$$
 where c and E are as in Lemma \ref{lemma}. \\
 Now, we claim that $\widetilde{\alpha_f}$ is nef. Indeed, it follows from Lemma \ref{lemma} that 
 $$1-2\sigma=\frac{1}{\lambda}$$
and by the representations (\ref{rep1}), (\ref{rep2}) and (\ref{rep3}) we infer that $\nu(T^{min}_{\widetilde{\alpha_f}},y)=0$ for every $y\in Y.$ Hence, the claim follows. \\  \indent
If $|S|=3$ then by Proposition \ref{2} $E_{nn}(\alpha_f)$ has $3$ components which are pairwise disjoint complex lines in $X$. In this case, we define the complex manifold $Y$ to be $X$ blown up along each component of $E_{nn}(\alpha_f)$ successively and apply the above analysis to drive the assertion. We omit the details of this part.
\end{proof}

 \bibliographystyle{alpha}
\bibliography{biblio}
\end{document}